\documentclass[letterpaper, 11pt, reqno]{amsart}
\usepackage[margin=1.2in,marginparwidth=1.5cm, marginparsep=0.5cm]{geometry}
\usepackage{amsmath, mathrsfs, amssymb,amscd,amsthm,amsxtra, esint}
\usepackage{tikz} 
\usepackage{color}

\usepackage[implicit=true]{hyperref}

\usepackage{cases}  
\usepackage{upgreek} 

\usepackage{ulem} 
\usepackage[thicklines]{cancel}
\allowdisplaybreaks[2]

\definecolor{gr}{rgb}   {0.,   0.69,   0.23 }
\definecolor{bl}{rgb}   {0.,   0.5,   1. }
\definecolor{mg}{rgb}   {0.85,  0.,    0.85}
\definecolor{yl}{rgb}   {0.8,  0.7,   0.}
\definecolor{or}{rgb}  {0.7,0.2,0.2}

\usetikzlibrary{shapes.misc}
\usetikzlibrary{shapes.symbols}
\usetikzlibrary{shapes.geometric}

\tikzset{
	ddot/.style={circle,fill=white,draw=black,inner sep=0pt,minimum size=0.8mm},
	>=stealth,
	}

\tikzset{
	ddot2/.style={circle,fill=black,draw=black,inner sep=0pt,minimum size=0.8mm},
	>=stealth,
	}

\newtheorem{theorem}{Theorem} [section]

\newtheorem{lemma}[theorem]{Lemma}
\newtheorem{proposition}[theorem]{Proposition}
\newtheorem{remark}[theorem]{Remark}

\newtheorem{definition}[theorem]{Definition}
\newtheorem{corollary}[theorem]{Corollary}

\newtheorem{oldtheorem}{Theorem}

\DeclareMathOperator{\com}{com}

\DeclareMathOperator{\Id}{Id}

\DeclareMathOperator{\Ker}{Ker}

\newcommand{\I}{\mathcal{I}}

\newcommand{\noi}{\noindent}
\newcommand{\Z}{\mathbb{Z}}
\newcommand{\R}{\mathbb{R}}

\newcommand{\T}{\mathbb{T}}
\newcommand{\bul}{\bullet}

\newcommand{\CC}{\mathcal{C}}

\renewcommand{\L}{\mathcal{L}}

\newcommand{\F}{\mathcal{F}}

\newcommand{\al}{\alpha}

\newcommand{\dl}{\delta}
\newcommand{\updl}{\updelta}

\newcommand{\Dl}{\Delta}
\newcommand{\eps}{\varepsilon}

\newcommand{\g}{\gamma}
\newcommand{\G}{\Gamma}
\newcommand{\ld}{\lambda}
\newcommand{\Ld}{\Lambda}
\newcommand{\s}{\sigma}

\newcommand{\ft}{\widehat}

\newcommand{\wt}{\widetilde}

\newcommand{\dx}{\partial_x}
\newcommand{\dt}{\partial_t}

\newcommand{\ta}{\theta}

\newcommand{\les}{\lesssim}

\newcommand{\jb}[1]
{\langle #1 \rangle}

\newcommand{\pa}{\partial}

\newcommand{\N}{\mathbb{N}}

\newcommand{\cL}{\mathcal{L}}
\newcommand{\cC}{\mathcal{C}}

\newcommand{\cX}{\mathcal{X}}

\newcommand{\Lip}{\mathrm{Lip}}
\newcommand{\uw}{U^w}
\newcommand{\uu}{\mathbf{u}}
\newcommand{\vv}{\mathbf{v}}
\newcommand{\z}{\zeta}

\newcommand{\Ta}{\Theta}

\newcommand{\KDV}{\text{\rm KdV} }

\newcommand{\too}{\longrightarrow}

\newtheorem*{ackno}{Acknowledgements}

\numberwithin{equation}{section}
\numberwithin{theorem}{section}

\makeatletter
\@namedef{subjclassname@2020}{\textup{2020} Mathematics Subject Classification}
\makeatother

\begin{document}
\baselineskip = 14pt

\title[Refined GWP of  the modulated KdV]{Refined global well-posedness for the periodic modulated Korteweg-de Vries equation}

\author[D.~Greco, M.~Gubinelli, S.~Liu, and T.~Oh]
{Damiano Greco, Massimiliano Gubinelli, Shao Liu,  and Tadahiro Oh}

\address{Damiano Greco, School of Mathematics\\
The University of Edinburgh\\
and The Maxwell Institute for the Mathematical Sciences\\
James Clerk Maxwell Building\\
The King's Buildings\\
Peter Guthrie Tait Road\\
Edinburgh\\ 
EH9 3FD\\
United Kingdom}

\email{dgreco@ed.ac.uk}

\address{Massimiliano Gubinelli, Mathematical Institute\\ 
University of Oxford\\ 
United Kingdom}

\email{gubinelli@maths.ox.ac.uk}

\address{
Shao Liu, Mathematical Institute\\
University of Bonn\\
Endenicher Allee 60\\
53115\\
Bonn\\
Germany}

\email{shaoliu@math.uni-bonn.de}

\address{
Tadahiro Oh, 
School of Mathematics\\
The University of Edinburgh\\
and The Maxwell Institute for the Mathematical Sciences\\
James Clerk Maxwell Building\\
The King's Buildings\\
Peter Guthrie Tait Road\\
Edinburgh\\
EH9 3FD\\
 United Kingdom\\
and School of Mathematics and Statistics, Beijing Institute of Technology, Beijing 100081, China}
%

%

\email{hiro.oh@ed.ac.uk}

\subjclass[2020]{35Q53, 60H15,  60H50, 60L90}

\keywords{modulated dispersion;   Korteweg-de Vries equation;
Young integral; sewing lemma; global well-posedness; $I$-method}

\begin{abstract}
We revisit the pathwise global well-posedness issue of 
 the modulated Korteweg-de Vries equation (KdV) on the circle.
In the previous work (2024), 
by combining the $I$-method and the sewing lemma, 
the second and fourth authors with C.\,Chouk, G.\,Li, and J.\,Li
proved its global well-posedness in negative Sobolev spaces.
This result was, however, restricted to 
 the scaling subcritical regime
$s > - \frac 32$
due to the use of 
the classical KdV scaling.
In this paper, by noting that the modulated KdV
enjoys 
additional one degree of  freedom in its scaling symmetry
 thanks to the modulation term, 
 we apply  a non-KdV scaling to the unknown
and 
prove that, given  
 {\it any} $s \in \mathbb R$,  the modulated KdV on the circle with a sufficiently irregular modulation is globally well-posed in $H^s(\T)$, 
thus  going beyond the barrier of the scaling critical 
 regularity $s = - \frac 32$.

\end{abstract}

%
\maketitle
%


\tableofcontents

\newpage

\section{Introduction} 
\subsection{Modulated Korteweg-de Vries equation} 
\label{SSEC:moKdV}

In this paper, we study global well-posedness of 
 the following modulated Korteweg-de Vries equation (KdV) on the circle
$\T=\R/(2\pi\Z)$:
\begin{align} 
\begin{cases}
    \dt u + \dx^3 u \cdot \dt w = \dx u^2 \\
    u|_{t=0} = u_0,
\end{cases} 
\label{kdv1}
\end{align}

\noi
where 
$u \colon \R\times \T \to \R$  is the unknown and 
 $w \colon \R\to\R$ is a continuous function of time, 
called a {\it modulation}. 
The modulated KdV \eqref{kdv1} naturally appears as a model for weakly nonlinear long waves in an inhomogeneous waveguide; see \cite{CMG, HZ}, where the modulation $w$ is taken to be periodic but not differentiable.

In \cite{DD, DT}, 
de Bouard,  Debussche, and Tsutsumi
used stochastic calculus to study 
the modulated nonlinear Schr\"odinger equation (NLS), 
where the modulation function $w$ is given by a Brownian motion; see also \cite{DR2,Ste}. 
For a general modulation function $w$, however, such an approach based on stochastic calculus is not available.
In \cite{CG1, CGLLO}, C.\,Chouk and the second author 
with their coauthors
developed a novel pathwise approach to study 
modulated dispersive equations, exploiting 
the following notion of 
{\it irregularity} of the modulation function~$w$
introduced in~\cite{CG16}.

\begin{definition}
\label{DEF:ir}
\rm
Let $\rho > 0$ and $0 < \gamma < 1$. Given $T > 0$, we say that a function $w \in C( [0, T]; \R)$ is $(\rho, \gamma)$-irregular on the time interval $[0, T]$ if we have 
\begin{align} 
\label{rho1}
    \| \Phi^{w} \|_{\mathcal{W}^{\rho, \gamma}_{T}} := \sup_{a \in \R} \sup_{0 \leq r < t \leq T} \jb{a}^{\rho} \frac {|\Phi^{w}_{t, r} (a)|}{|t -r|^{\gamma}} < \infty, 
\end{align}

\noi 
where 
\begin{align} 
\label{rho2}
    \Phi^w_{t, r} (a) = \int_{r}^{t} e^{i a w (t')} dt'. 
\end{align} 

\noi 
We say that $w$ is $(\rho, \gamma)$-irregular on $\R_{+}$ if it is $(\rho, \gamma)$-irregular on $[0, T]$ for each finite $T>0$. 
\end{definition}

See Remark \ref{REM:BM1}
for important examples of irregular paths and 
a further discussion.

\medskip

Let us briefly go over the main ideas of the pathwise approach in \cite{CG1, CGLLO}
by restricting our attention to the modulated KdV \eqref{kdv1}.
We first note that 
the equation \eqref{kdv1}
is merely formal since the 
 time derivative of the modulation  $w$ does not exist in general.
  To bypass this problem, we consider the following Duhamel formulation (= mild formulation) of \eqref{kdv1}, 
  involving only the modulation $w$ and not its time derivative: 
\begin{align}
\label{mild1} 
    u (t) = U^{w} (t) u_0 + U^{w} (t) \int_{0}^{t} U^{w} (t')^{-1} \pa_x \big(u (t')^{2} \big) dt', 
\end{align} 

\noi 
where $U^{w} (t) = e^{- w (t) \pa_x^3}$ denotes the modulated linear propagator.
Here, 
 we impose $w (0) = 0$ such that $U^{w} (0) = \rm{Id}$.\footnote{Note that the normalization $w(0) = 0$ is not an additional restriction since only the time derivative $\dt w$
 appears in the modulated equation \eqref{kdv1}.}  
Let $\uu$
denote  the 
 modulated interaction representation
 of the unknown $u$ defined by 
\begin{align*}
\uu (t) = U^{w} (t)^{-1} u(t).
\end{align*}

\noi 
Then, 
the Duhamel formulation \eqref{mild1}
becomes
\begin{align}
\label{mild3} 
\uu (t) = u_0 + \int_{0}^{t} \uw (t')^{-1} \pa_{x} \big( ( \uw (t') \uu (t'))^2 \big) d t'. 
\end{align} 

\noi
Thanks to the absence of $\uw (t)$
in front of the time integral (as compared to \eqref{mild1}), 
the integral equation \eqref{mild3} allows
us to exploit temporal regularity 
of $\uu$
(just as in 
the Fourier restriction norm method \cite{BO93}). 
The main novelty in \cite{CG1, CGLLO}
is to give a meaning
to 
the integral term in~\eqref{mild3} 
as a nonlinear Young integral 
$\I^{X}(\uu)$ 
associated
with the bilinear driver~$X$
defined in~\eqref{bd0} (see also \eqref{Xld1}), 
thus reducing \eqref{mild3}
to the following 
 nonlinear Young differential equation (YDE):
\begin{align}
    \uu = u_0 + \I^{X} (\uu) 
\label{YDE0}
\end{align} 

\noi 
whose local well-posedness follows
from a simple contraction argument;
see Subsection~\ref{SUBSEC:YDE}
for a brief 
review on the construction of nonlinear Young integrals
and local well-posedness of a  nonlinear YDE
of the form  \eqref{YDE0}.
See also 
\cite[Section 3]{CGLLO}
and \cite[Section 4]{CGLLO2}
for an extensive review  on the subject.
The actual construction of the nonlinear Young integral
$\I^{X}(\uu)$
in \cite{CG1, CGLLO}
was based on 
 the sewing lemma (Lemma \ref{LEM:sew})
due to  the second author \cite{Gub04}, 
reducing the matter to studying
 mapping properties of the bilinear driver $X$
for the modulated KdV; see Lemma \ref{LEM:tri1}
(with $\ld  = 1$).

We now recall
the local and global well-posedness results
for the modulated KdV \eqref{kdv1} on~$\T$;
see
\cite[Theorem 1.3]{CGLLO}.
See Remark \ref{REM:sol1}
for the precise meaning of a solution.

\begin{oldtheorem}
\label{THM:old}
Given $\rho \ge\frac 12$,  $\frac12< \g < 1$, and $T> 0$, 
let  $w$ be $(\rho,\g)$-irregular on $[0, T]$.

\smallskip

\noi
\textup{(i)}
Suppose that $\rho \ge \frac 12$ and $s\in \R$
satisfy one of the following conditions\textup{:}
\begin{align}
\begin{split}
\textup{(i.a)} &\ \  \tfrac 12 \le \rho \le \tfrac 34 \quad \text{and} \quad 
s > \tfrac 32 - 3 \rho, \\
\textup{(i.b)} &\ \  \rho > \tfrac 34\quad  \text{and} \quad  s \ge - \rho.
\end{split}
\label{reg1}
\end{align}

\noi
Then, the modulated KdV equation \eqref{kdv1}
on  $\T$
is semilinearly locally well-posed in $H^s(\T)$.

\smallskip

\noi
\textup{(ii)}
In addition, suppose that 
  $w$ is $(\rho,\g)$-irregular on $\R_+$
and  that $\rho > \frac 12$ and $s\in \R$
satisfy one of the following conditions\textup{:}
\begin{align}
\begin{split}
\textup{(ii.a)} &\ \ 
\tfrac 12 < \rho \le \tfrac 3{4\g}
 \quad \text{and} \quad 
s > \tfrac {3- 6\rho}{6-4\g}, \\
\textup{(ii.b)} &\ \ 
\tfrac 3{4 \g}< \rho < \tfrac 32
\quad  \text{and} \quad  s \ge - \rho,\\
\textup{(ii.c)} &\ \ 
  \rho \ge \tfrac 32
\quad  \text{and} \quad 
s  > -\tfrac 32.
\end{split}
\label{s1}
\end{align}

\noi
Then, 
 the modulated KdV equation \eqref{kdv1}
on $\T$
is globally  well-posed in $H^s(\T)$.

\end{oldtheorem}

Theorem \ref{THM:old}\,(i)
exhibited a novel {\it regularization-by-noise} phenomenon;
given {\it any} $s \in \R$, 
the modulated KdV \eqref{kdv1} on $\T$
is locally well-posed
in $H^s(\T)$
by taking $\rho \gg1$ sufficiently large, 
whereas 
 the (unmodulated) KdV equation on $\T$:
\begin{equation}
\dt u+  \dx^3 u  =\dx u^2
\label{kdv2}
\end{equation}

\noi
is known to be ill-posed in $H^s(\T)$ for $s < -1$;
see
\cite{KPV96, CKSTT03, 
KT1, M12,  KV19}
for the known well-\,/\,ill-posedness
results for the (unmodulated) KdV \eqref{kdv2} on $\T$.
In the context of stochastic differential equations 
and stochastic parabolic PDEs, 
there  have been intensive research activities
on regularization by noise 
since the 70's;
see 
 \cite{FGP, GG0, G22, RT, CGLLO} and the  references therein.
See also   survey works \cite{Flan, Gess}.
Prior to the work~\cite{CGLLO}, 
however, 
regularization by noise 
in the context of 
dispersive PDEs
was known only 
for probabilistic well-posedness with  random initial data and\,/\,or additive noises of super-critical regularity
(such as \cite{BO96, BT08, CO, BOP2, Poc, OP, GKO2, DNY2, DNY3, OOT1, Bring, OOT2, BDNY});
see \cite[Remark~1.8]{CGLLO}.
See also 
Remark~\ref{REM:RN1}.

Theorem \ref{THM:old}\,(ii) on 
global well-posedness
of the modulated KdV 
was based on the so-called $I$-method (= method of almost conservation laws)
introduced in \cite{CKSTT03}, 
combined with the sewing lemma.
On the one hand, when the modulation is sufficiently irregular, 
the modulated KdV \eqref{kdv1} is globally well-posed
below $H^{-1}(\T)$, where the (unmodulated) KdV~\eqref{kdv2}
is ill-posed.
On the other hand, the regularity range in~\eqref{s1}
is restricted to the scaling subcritical regime
$s > s_\text{crit} = - \frac 32$, 
where $ s_\text{crit} = - \frac 32$ denotes the scaling critical 
regularity for the (unmodulated) KdV \eqref{kdv2}.
This restriction $s > - \frac 32$ in Theorem \ref{THM:old}\,(ii)
comes from the use of the classical KdV scaling (see \eqref{scaling0})
to the unknown $u$
in the application of the $I$-method.

Our goal in this paper is to go beyond the barrier of the scaling critical regularity 
 $s_\text{crit} = - \frac {3}{2}$ by introducing a new one-parameter family of scaling transforms
 for the pair $(u, w)$
 in \eqref{kdv1}
  (see \eqref{scaling1} and \eqref{scaling2}), 
which    allows 
us to prove
 that, given  
 {\it any} $s \in \mathbb R$,  the modulated KdV \eqref{kdv1} on~$\T$ with a sufficiently irregular modulation is globally well-posed in $H^s(\T)$;
see Theorem~\ref{THM:main}.
This significantly improves
Theorem \ref{THM:old}\,(ii), 
thus 
establishing a strong
regularization-by-noise
phenomenon
even for global well-posedness
of the modulated KdV on~$\T$.

\begin{remark}\label{REM:R1}\rm
In \cite{CGLLO}, 
the authors also studied well-posedness
of the modulated KdV~\eqref{kdv1} on the real line, 
where they proved
that, for a sufficiently irregular modulation, 
the modulated KdV \eqref{kdv1} on the real line
is 
locally well-posed
in $H^s(\R)$ for $s > - \frac 32$, 
going below $H^{-1}(\R)$, 
where 
the (unmodulated) KdV \eqref{kdv2}
is known to be ill-posed.
See \cite[Theorem~1.7]{CGLLO}.
We, however, point out that, 
in stark contrast to the periodic case (Theorem~\ref{THM:old}\,(i)),
 this local well-posedness
result on the real line 
was restricted to the 
scaling subcritical regime
$s > - \frac 32$
due to the `high $\times$ high into low' frequency interactions
(which do not exist in the periodic setting).
In fact, it was shown that the solution map
fails to be $C^2$ on $H^s(\R)$
for $s < - \frac 32$ in the spirit of \cite{BO97, Tzv}, 
stating that one can not use a contraction argument
to study 
local well-posedness of the modulated KdV \eqref{kdv1}
in 
$H^s(\R)$
for $s < - \frac 32$;
see \cite[Remarks 1.8 and 4.5]{CGLLO}.
It is for this reason that we restrict
our attention to the modulated KdV \eqref{kdv1}
on the circle
in discussing refined global well-posedness
in this paper.

\end{remark}

\begin{remark}\label{REM:BM1}\rm

We recall 
the following results from \cite{CG16, GG}, 
providing important examples of $(\rho, \g)$-irregular paths:

\smallskip

\begin{itemize}
\item
[{\rm (i)}]
Let $\{W_t\}_{t\in \R_+}$ be a fractional Brownian motion 
of Hurst parameter $H\in(0,1)$.
Then, 
for any $\rho < \frac{1}{2H}$,  
there exists $\frac 12 < \g < 1$
such that,  with probability one,  the sample paths of 
$W$ are $(\rho,\g)$-irregular on $\R_+$.

\medskip

\item 
[{\rm (ii)}]
Let  $d \ge 1$.
Given  any $\delta \in (0,1)$, a generic  $\dl$-H\"older continuous function $w \in C^\delta([0,1];\R^d)$ 
 is $(\rho,\g)$-irregular for any $\rho < \frac1{2\delta}$ with some $\gamma = \gamma(\rho) \in (\frac 12,1)$.

\end{itemize}

\smallskip

\noi
The term   ``generic''
in (ii)
 is to be understood according to the notion of {\it prevalence}; see~\cite{GG} for details and the  references therein.

As pointed out in \cite[Corollary 1.5 and Remark 1.6]{CGLLO}, 
we note that, 
Theorem \ref{THM:old}
and our main result (Theorem \ref{THM:main})
apply 
to the case, where the modulation $w$ is given by a fractional Brownian motion
or a generic $\dl$-H\"older continuous function.
For further studies on the notion of irregularity, 
we refer readers to 
 recent works \cite{GG, RT}.

\end{remark}

\begin{remark}\label{REM:RN1} \rm

In \cite{CGLLO}, 
the authors also studied other modulated dispersive PDEs
such as 
 the modulated Benjamin-Ono equation 
and the stochastic modulated KdV with an additive forcing
and established new regularization-by-noise phenomena.
 See \cite{CGLLO} for a further discussion on various examples. We also mention recent works \cite{Tanaka, Robert1, Robert2, Robert3, Robert4} on pathwise well-posedness of various modulated dispersive equations; see \cite[Remark 1.15]{CGLLO}.

In a recent preprint
\cite{CGLLO2} with A.\,Chapouto, G.\,Li, and J.\,Li, 
the second and fourth authors
investigated pathwise well-posedness of stochastic modulated dispersive equations with 
multiplicative noises.
By combining the nonlinear Young integration theory
(as in \cite{CGLLO} and this paper)
with a novel pathwise well-posedness 
approach for (unmodulated) dispersive PDEs
with multiplicative noises, 
developed in 
\cite{CLO2, CLO3, COZ, CLOO} (see  also \cite{OWZ, SLO}), 
they established a new regularization-by-noise phenomenon. For example, for the stochastic modulated KdV with a multiplicative fractional-in-time noise 
(with the Hurst parameter $\frac 12 < H < 1$), 
they proved
that, given any $s \in \R$ and $\s \in \R$ of the spatial regularity of the noise, 
it is locally well-posed in $H^s(\T)$, 
provided that the modulation is sufficiently irregular.

In a recent work \cite{GLLO}
 with  G.\,Li and J.\,Li, 
the second and fourth authors
 implemented a  normal form method to study modulated dispersive PDEs,
 establishing unconditional uniqueness of solutions.
This normal form
approach provides an alternative way to give a meaning to the 
integral term in~\eqref{mild3} (and 
$ \I^X(\uu)$ in \eqref{YDE0})
without relying on the sewing lemma 
and 
without assuming any positive temporal regularity
on the modulated interaction representation $\uu$, 
which 
can be viewed,  
at a philosophical level, 
as 
a controlled path approach
to nonlinear Young integration
which 
significantly improves any existing theory
in terms of the temporal regularity.
See \cite{GLLO}
for a further discussion.


\end{remark}

\subsection{Refined global well-posedness}
\label{SUBSEC:1.2}

In this subsection, we introduce a new one-parameter family
of scaling transforms for the modulated KdV \eqref{kdv1} on $\T$
and state a refined global well-posedness result (Theorem \ref{THM:main}).

In \cite{CGLLO}, 
the authors established
Theorem \ref{THM:old}\,(ii) 
on global well-posedness of the modulated KdV \eqref{kdv1} on $\T$
by applying the $I$-method together with the sewing lemma.
In this argument, the following scaling property
of the modulated KdV \eqref{kdv1}
plays a crucial role.
Given a modulation function~$w$, 
let $u$ be a solution to \eqref{kdv1} on $[0, T]\times \T$
(for some $T > 0$)
with initial data $u|_{t = 0} = u_0$.
Given $\ld \ge 1$, apply the classical KdV scaling to $u$:
\begin{align} 
    u^{\ld} (t, x) = \ld^{-2} u (\ld^{-3} t, \ld^{-1} x).
\label{scaling0} 
\end{align} 

\noi
Then,
$u^\ld$ is a solution to 
 the following scaled modulated KdV on $[0, \ld^3 T] \times \T_{\lambda}$, 
where   $\T_{\lambda} = \R/ (2 \pi \lambda \Z)$ denotes the dilated circle:
 \begin{align} 
\label{kdv4}
    \dt u^{\lambda} + \dx^3 u^{\lambda} \cdot \dt w^{\lambda} = \dx (u^{\lambda})^2
\end{align}

\noi
with the scaled initial data
\begin{align*}
u_0^\ld(x) = \ld^{-2} u_0(\ld^{-1}x), 
\end{align*}

\noi
where the scaled modulation $w^\ld$ is given by
\begin{align}
    w^{\ld} (t) = \ld^{3} w (\ld^{-3} t)
\label{scaling0a}
\end{align}

\noi
such that $\dt w^\ld(t) = \dt w(\ld^{-3} t)$.
See \cite[Lemma 7.3]{CGLLO}
for the equivalence of the unscaled and scaled modulated KdV equations.

We now introduce a new one-parameter  family of scaling transforms
for the modulated KdV \eqref{kdv1} on $\T$.
Given $b \in \R$ and $\ld \ge 1$,
define $u^\ld$ and $w^\ld$ by setting
\begin{align} 
\label{scaling1}
    u^{\lambda} (t, x) = \lambda^{- b + 1} u (\lambda^{-b} t, \lambda^{-1} x)
\end{align}

\noi
and 
\begin{align} 
\label{scaling2}
    w^{\lambda} (t) = \lambda^{3} w (\lambda^{-b} t)
\end{align} 

\noi
such that $\dt w^\ld(t) = \ld^{3-b} \dt w(\ld^{-b} t)$.
Then, we see that the following equivalence
holds; 
a pair $(u, w)$
satisfies  the modulated KdV \eqref{kdv1} on $[0, T] \times \T$ with
 initial data $u|_{t = 0} = u_0$ if and only if 
the scaled pair $(u^\ld, w^\ld)$
satisfies
 the  scaled modulated KdV 
\eqref{kdv4}
on $[0, \lambda^b T] \times \T_{\lambda}$
with the  scaled initial data 
\begin{align} 
\label{scaling4}
    u_0^{\lambda} (x) = \lambda^{-b + 1} u_0 (\lambda^{-1} x).
\end{align} 

\noi 
We note that, when  $b = 3$, 
\eqref{scaling1} and \eqref{scaling2}
reduce to the classical KdV scaling
 \eqref{scaling0} and~\eqref{scaling0a}, 
 respectively.

The main observation 
is that as compared to the (unmodulated) KdV \eqref{kdv2}, 
the modulated KdV 
\eqref{kdv1} enjoys one more degree of freedom
thanks to the presence of the modulation $w$, 
which we exploited in defining
the new one-parameter family of the scaling transforms~\eqref{scaling1}
and~\eqref{scaling2}.
By applying the $I$-method and the sewing lemma 
as in \cite[Section~7]{CGLLO}
with the new scaling transforms
\eqref{scaling1}
and~\eqref{scaling2}
for a  suitable choice of $b > 0$
(see~\eqref{scaling5}), 
we establish the following refined global well-posedness
of the modulated KdV \eqref{kdv1} on $\T$.

\begin{theorem} 
\label{THM:main}
Given $\rho > \frac 12$ and  $\frac 12 < \gamma < 1$, let $w$ be $(\rho, \gamma)$-irregular on $\R_{+}$. Suppose that  $s\in \R$ satisfies one of the following conditions\textup{:}
\begin{align}
\begin{split}
\textup{(i)} & \ \  \tfrac{1}{2} < \rho \le \tfrac{3\g}{6\g-2} \quad \text{and} \quad s > \big(\tfrac{3}{2}-3\rho \big) \g, \\ 
\textup{(ii)} & \ \   \tfrac{3\g}{6\g-2}<\rho \leq \tfrac {3}{2}\quad \text{and} \quad s \ge -\rho,  \\ 
\textup{(iii)} & \ \ \tfrac{1}{2}<\g<\tfrac{\sqrt{5}-1}{2},\\
&\ \textup{(iii.a)} \  \ \tfrac{3}{2}\le\rho<\tfrac{3(1-\g)^2}{2(-\g^2-\g+1)}, \quad \text{and} \quad s \ge -\rho,  \\  
& \text{or} \\ 
&\ \textup{(iii.b)} \ \ \rho \geq \tfrac{3(1-\g)^2}{2(-\g^2-\g+1)}, \quad \text{and} \quad s > -\tfrac {4\rho(2\g-1)+3(1-\g)^2}{2(-\g^2+3\g-1)}, \\
\textup{(iv)} & \ \ \tfrac {\sqrt{5} - 1}{2}\le  \gamma < 1,\quad \rho\ge \tfrac{3}{2},\quad \text{and}\quad s\ge -\rho. 
\end{split} 
\label{mcon1}
\end{align}
\noi 
Then, the modulated KdV equation \eqref{kdv1} on $\T$ is globally well-posed in $H^s (\T)$. 
\end{theorem}

We note that, when $\tfrac{1}{2}<\g<\tfrac{\sqrt{5}-1}{2}$,  
\[-\g^2+3\g-1 > -\g^2-\g+1 > 0.\]

From $(\text{iii.b})$ and $(\text{iv})$ in \eqref{mcon1}, 
we observe a very strong regularization effect: given \textit{any} $s \in \R$, 
the modulated KdV~\eqref{kdv1} on $\T$ with a sufficiently irregular modulation 
$w$ (that is, $\rho$ is sufficiently large) is globally well-posed in $H^s (\T)$. 
This is in sharp contrast to Theorem \ref{THM:old}\,(ii), where global well-posedness holds only 
in the scaling subcritical regime $s > -\frac{3}{2}$, regardless of how large $\rho$ is. 
We also note that Theorem \ref{THM:main} improves Theorem \ref{THM:old}\,(ii) in all cases.


We present  a proof of Theorem~\ref{THM:main}
in Section \ref{SEC:GWP}.
For
 $s \ge 0$, global well-posedness of the modulated KdV \eqref{kdv1}
follows from  
  the conservation of the $L^2$-norm under  \eqref{kdv1}
  (see \cite[Proposition 6.1]{CGLLO})
and the persistence of regularity for  \eqref{kdv1}
(see \cite[Remark 1.4\,(iii)]{CGLLO}).
Hence, we restrict our attention to   
 the case $s < 0$
in the following.
In view of the invariance of 
the modulated KdV \eqref{kdv1}
under the following Galilean transform:
$u(t, x)\mapsto u (t, x - 2\al_0t)-\al_0$
along with the conservation of the spatial mean (which can be verified
by arguing as in \cite[Section 6]{CGLLO}), 
we assume that 
\begin{align}
\begin{split}
&  \text{a solution $u(t)$ to the modulated KdV \eqref{kdv1} has spatial mean~$0$}\\
& \hspace{2.3cm} \text{for any $t \in \R$ as long as it exists}
\end{split}
\label{mean0}
\end{align}

\noi
in the 
remaining part of this paper. 
Lastly, 
for simplicity of the presentation, 
we only consider positive times in the following.

\begin{remark}\rm
(i)
In applying the $I$-method, 
we first apply the scaling~\eqref{scaling1}, 
which leads to the condition $b > \frac 32- s$; see \eqref{scaling3} and \eqref{scaling5}.
In particular, this implies that, given $s < 0$ with $|s| \gg 1$, 
we need to take $b \gg 1$.
As such, 
we restrict our attention to the case
$ b > \frac {3}{2 (1 - \g)}$
for conciseness of the presentation, 
since the case 
$ b \le \frac {3}{2 (1 - \g)}$ requires a slightly different consideration, 
while the latter can not handle the case $b \gg 1$ (corresponding to $|s| \gg1$, 
which is the case we are after in this paper).
See 
Remark \ref{REM:tau}
for a further explanation.

\smallskip

\noi
(ii) 
In view of 
\eqref{scaling3}, 
it seems that 
the new scaling transforms \eqref{scaling1} and \eqref{scaling2}
make 
the well-posedness problem of the modulated KdV \eqref{kdv1}
in $H^s(\T)$  (and in $H^s(\R)$)
scaling subcritical by choosing $b > \frac 32 - s$; see \eqref{scaling5}.
Indeed, these new scaling transforms
allowed us to significantly improve
global well-posedness regularity thresholds
in Theorem \ref{THM:main}.
However, they do not seem to allow us to 
improve local well-posedness
of the modulated KdV \eqref{kdv1}
on the real line
from \cite{CGLLO}, which was
restricted
to the scaling subcritical regime $s > - \frac 32$.
See Remark \ref{REM:R1}.

\end{remark}

\begin{remark}\rm

In a recent preprint \cite{AGT},
T.\,Arthur, K.\,Tsugawa, 
and the first author 
studied the well-posedness issue
of the following coupled system of modulated KdV
on $\T$:
\begin{align} 
\begin{cases}
\dt u + \dx^3 u \cdot \dt w = \frac 12 \dx v^2 \\
\dt v + \al \dx^3 v \cdot \dt w = \dx (uv) 
\end{cases} 
\label{kdv3}
\end{align}

\noi
for $0 < \al \le 4$.

In the unmodulated case (i.e.~$w(t)  \equiv t$), 
the fourth author 
\cite{Oh1, Oh2, Oh3}
studied the well-posedness
issue of the coupled KdV system \eqref{kdv3}
and proved that, when $\al \ne 1$, 
the threshold regularities for local and global well-posedness
depend sensitively on the Diophantine 
properties
of certain constants derived from the coupling parameter $\al$.
We note that these regularity thresholds in \cite{Oh1, Oh2} 
were in general much worse (more than one regularity) than 
the KdV case (and the $\al = 1$ case).

In the modulated setting, 
by combining the $I$-method and the sewing lemma, 
utilizing  the 
 one-parameter  family of scaling transforms
\eqref{scaling1} and \eqref{scaling2}, 
with the Diophantine
properties
of certain constants, depending on $\al$ as in \cite{Oh1, Oh2},
the authors 
in \cite{AGT}
 proved that,
for {\it almost every} $\al \in (0, 4]$
and given any $s \in \R$, 
the system 
 \eqref{kdv3}
is globally well-posed in $H^s(\T) \times H^s(\T)$, 
provided that the modulation $w$ is sufficiently irregular.
See \cite{AGT} for a further discussion.

\end{remark}

\section{Notations and preliminary materials}

In this section, we recall some notations and preliminary results. 
We first introduce basic notations
and function spaces.
We then provide a brief review on theory of nonlinear Young integration
and local well-posedness of a nonlinear YDE.

\subsection{Basic notations}
Let $A\les B$ denote an estimate of the form $A\leq CB$ for some constant $C>0$. We write $A\sim B$ if $A\les B$ and $B\les A$, while $A\ll B$ denotes $A\leq c B$ for some small constant $c> 0$. 
We may write  $\ll_{\al}$ and $\sim_{\al}$ to 
emphasize the dependence on an external parameter $\al$.
We use $C>0$ to denote various constants, which may vary line by line.

In expressing the dependence of a function $u$
on the time variable, we often use the short-hand notation
$u_t = u(t)$,  which is standard in probability theory and stochastic analysis.

We often drop
 dependence on certain parameters such as $\rho$ and $\g$ (and  on $s$ and $b$).

%

\subsection{Function spaces}

Given  $\ld \ge 1$, 
set
 $\Z_{\ld} = \Z/\ld$
and 
$\Z^*_{\ld} = \Z_\ld\setminus\{0\}$. 
When $\ld = 1$, we simply set 
$\Z^* = \Z^*_{\ld = 1}$.
We define 
 the Fourier transform of a function $f$ on $\T_\ld$
by 
$$
\F_{\T^{\ld}} (f) (n) = \ft  f(n)=\int_{\T_\ld}f(x)e^{-i n x} \frac{d x}{2\pi}
$$

\noi
for $n\in\Z_{\ld}$.
Then, we have 
\begin{align}
\begin{split}
f(x)
& =
\frac{1}{\ld} \sum_{ n \in\Z_{\ld}}\ft  f(n)e^{i  nx}, \\
\int_{\T_\ld}|f(x)|^2 \frac{d x}{2\pi}
& =\frac{1}{\ld}\sum_{n\in\Z_{\ld}}|\ft  f(n)|^2,\\
\ft{fg}(n)
& =\frac{1}{\ld}
 \sum_{ \substack{n_1, n_2 \in \Z_\ld\\n = n_1+n_2}}
 \ft  f(n_1)\ft  g(n_2).
\end{split}
\label{FT3}
\end{align}

\noi
Given $s \in \R$, 
we define the non-homogeneous and homogeneous Sobolev spaces $H^{s}(\T_\ld)$ 
and $\dot H^{s}(\T_\ld)$ 
via the (semi-)norms: 
\begin{align}
\begin{split}
\|  f  \|^2_{H^{s}(\T_\ld)}
& =\frac{1}{\ld}
\sum_{n \in \Z_{\ld}}
\jb{n}^{2s} | \ft f(n) |^2, \\
\|  f  \|^2_{\dot H^{s}(\T_\ld)}
& =\frac{1}{\ld}
\sum_{n \in \Z_{\ld}^*}
|n|^{2s} | \ft f(n) |^2.
\end{split}
\label{FT4}
\end{align}

\medskip

Given $k \in \N$, 
let $V, V_1, \dots, V_k$  be separable Hilbert spaces.
We use 
\begin{align*}
\cL_k\Big(\bigotimes_{j = 1}^k V_j; V\Big)
\end{align*}

\noi
 to denote 
the Banach space of bounded $k$-linear operators 
on $\bigotimes_{j = 1}^k V_j$ 
(equipped with the Hilbert tensor norm) 
with values in $V$.
When $V_j =  V$ for $j =1, \dots,k$,
we simply set  
 $\cL_k(V)=\cL_k(V^{\otimes k}; V)$.

Let $V$ be a Banach space and $T>0$.
For $n\in\N$, we denote 
\begin{align*}
\Delta_{n, T} = 
\big\{ (t_1, \ldots, t_n) \in [0,T]^n: \ t_i > t_j 
\text{ for } i < j\big\}.
\end{align*}
We denote by $C_{n,T}V$ 
the space of continuous functions 
from $\Delta_{n,T}$ to $V$. When $n=1$,
we may write $C_T V$ for simplicity, 
and equip this space with the supremum norm: 
\begin{align*}
\|f\|_{C_T V} = \|f\|_{L^\infty_T V} = \sup_{0\leq t \leq T} \|f(t)\|_V.
\end{align*}

\noi
We define the coboundary operator 
$\updl: C_{n,T} V  \to C_{n+1,T} V$ 
as follows; 
given  $f\in C_{n,T} V$ 
and $(t_1,\ldots, t_{n+1}) \in \Dl_{{n+1},T}$, 
we set
\begin{align*}
(\updl f)_{t_1,\ldots , t_{n+1}} = 
\sum_{k=1}^{n+1} (-1)^{k} f_{t_1, \ldots, t_{k-1}, t_{k+1}, \ldots, t_{n+1}}.
\end{align*}

\noi
For example, for $f\in C_T V$
and 
$g\in C_{2,T} V$, 
we have 
\begin{align*}
\begin{split}
(\updl f )_{t,r} &= f_t - f_r, \\
(\updl g)_{t_1,t_2,t_3} &= 
g_{t_1,t_3} - g_{t_1,t_2} - g_{t_2,t_3}
\end{split}
\end{align*}

\noi
for $(t,r)\in\Dl_{2,T}$
and $(t_1,t_2,t_3) \in \Dl_{3,T}$.
As noted in \cite{GT10}, 
the sequence
\begin{align*}
0 \too V \too C_{1, T}V \stackrel{\updl}{\too} C_{2, T}V
\stackrel{\updl}{\too} C_{3, T}V
\stackrel{\updl}{\too} \cdots
\end{align*}

\noi
is exact. 
In particular, we have 
 $\updl\circ\updl =0$ and 
if $f \in C_{n,T} V$ with $\updl f =0$, 
then there exists a $g\in C_{n-1,T}V$ 
such that $f= \updl g$; 
see, for example,  \cite[Lemma 2.1]{GT10}.

Given $0 < \g < 1$, we denote by $C^\g_T V = C^\g([0, T]; V)$ the space of $\g$-H\"older continuous functions taking values in $V$, endowed  with the seminorm:
\begin{align*}
\|f\|_{C^\g_T V} = \sup_{(t,r)\in \Dl_{2,T}} 
\frac{\|(\updl f)_{t,r}\|_V}{|t-r|^\g}.
\end{align*}

\noi
We also define 
 $\CC^\g_T V = \CC^\g([0, T]; V)$ via the norm:
\begin{align*}
\| f \|_{\CC^\g_TV} = \| f \|_{L^\infty_TV} + \|f \|_{C^\g_T V}.
\end{align*}

\noi 
We also introduce the spaces 
$C^\g_{n,T}V$, $n=2,3$, 
equipped with the following H\"older-type norms;
 for $g\in C_{2,T}V$ and $h\in C_{3,T}V$, we set
\begin{align}
\begin{split}
\| g\|_{C^\g_{2,T} V} & 
= \sup_{(t,r) \in \Dl_{2,T}} \frac{\|g_{t,r} \|_{V} }{|t-r|^{\g}}, \\
\| h\|_{C^\g_{3,T} V} & 
= \inf_{0<\al<\g} 
\sup_{(t_1,t_2,t_3) \in \Dl_{3,T}} 
\frac{ \|h_{t_1,t_2,t_3}\|_V}
{|t_1-t_2|^{\al} |t_2-t_3|^{\g-\al}}.
\end{split}
\label{Ho2}
\end{align}

\medskip

Let $V$ and $W$ be Banach spaces.
Given $k \in \N$, 
we  use $\Lip_k(V; W)$ 
to denote the Banach space of 
locally Lipschitz maps $f:V\to W$
with polynomial growth of order $k$
such that 
\begin{align}
\|f\|_{\Lip_k(V; W)} = 
\sup_{x,y\in V} 
\frac{\|f(x)-f(y)\|_W}{\|x-y\|_V \big(1+\|x\|_V+\|y\|_V\big)^{k-1}}
<\infty .
\label{Lip1}
\end{align}

\noi
When $V = W$, 
we simply set $\Lip_k(V)= \Lip_k(V;V)$.
Given an integer $k \ge 2$, 
we say that $f\in\Lip^2_k(V; W)$ 
if 
\smallskip

\begin{itemize}
\item[(i)] $f\in\Lip_k(V; W)$,

\smallskip
\item[(ii)]
 $f$ is Fr\'echet differentiable 
with 
$Df \in \Lip_{k-1}(V;\cL_1(V; W))$.

\end{itemize}

\smallskip

\noi
From \eqref{Ho2} and \eqref{Lip1}, we have 
\begin{align}
\|f \|_{C^\g_{2, T}\Lip_k(V;W)}
= \sup_{(t,r) \in \Dl_{2,T}}
\frac 1{|t-r|^{\g}}
\sup_{x,y\in V} 
\frac{\| f_{t, r}(x)- f_{t, r}(y)\|_W}{\|x-y\|_V \big(1+\|x\|_V+\|y\|_V\big)^{k-1}}.
\label{Ho3}
\end{align}

\medskip

Given $s \in \R$, $0 < \g < 1$, $k \in \N$,  $T > 0$, and $\ld \ge 1$, 
we define the space 
$\cX^{s, \g}_k([0, T]\times\T_\ld)$ of drivers 
as follows; for $k \ge 2$, we set
\begin{align}
\begin{split}
& \cX^{s, \g}_k([0, T]\times \T_\ld) \\
& \quad = 
\big\{X \in C^\g_{2, T} \Lip_k^2(H^s(\T_\ld)):
X(0) = DX[0]= 0 \text{ and }\updl X = 0\big\}, 
\end{split}
\label{X1}
\end{align}

\noi
endowed with the norm:
\begin{align*}
\|X\|_{\cX^{s, \g}_k([0, T]\times \T_\ld)} 
=  \|X\|_{C^\g_{2, T}\Lip_k(H^s(\T_\ld))} + \|DX\|_{C^\g_{2, T}\Lip_{k-1}(H^s(\T_\ld);\cL_1(H^s(\T_\ld)))}.
\end{align*}

\noi
When $k = 1$, we set
\begin{align*}
 \cX^{s, \g}_1 ([0, T]\times \T_\ld) 
& =  \big\{ X \in 
C^\g_{2, T} \L_1(H^s(\T_\ld)):
\updl X = 0\big\}.
\end{align*}

\noi
In \eqref{X1}, $X(0) = 0$ (and $DX[0]= 0$) means $X_{t, r}(0) = 0$ 
(and $DX_{t, r}[0]= 0$, respectively)
for any $(t, r) \in \Dl_{2, T}$, 
where $DX_{t, r}[0]$
denotes the Fr\'echet derivative of $X_{t, r}$ at $u = 0 \in H^s(\T_\ld)$.

We also  define 
$\dot \cX^{s, \g}_k([0, T]\times\T_\ld)$ by 
replacing $H^s(\T_\ld)$ 
with $\dot H^s(\T_\ld)$
in the definition of 
$ \cX^{s, \g}_k([0, T]\times\T_\ld)$.
In particular, for $k \ge 2$, we have 
\begin{align} 
\begin{split}
& \dot \cX^{s, \g}_k ([0, T]\times \T_\ld) \\
& \quad = \big\{X \in C^\g_{2, T} \Lip_k^2(\dot H^s(\T_\ld)):X(0) = DX[0]= 0 \text{ and }\updl X = 0\big\}, 
\end{split}
\label{X1a}
\end{align}

\noi 
endowed with the norm:
\begin{align*}
\|X\|_{\dot \cX^{s, \g}_k([0, T]\times \T_\ld)} 
=  \|X\|_{C^\g_{2, T}\Lip_k(\dot H^s(\T_\ld))} + \|DX\|_{C^\g_{2, T}
\Lip_{k-1}(\dot H^s(\T_\ld);\cL_1(\dot H^s(\T_\ld)))}. 
\end{align*}

We may use short-hand notations such as
$\CC^\al_T H^s_x  = \CC^\al\big([0, T]; H^s(\T_\ld))$, etc.
 when there is no ambiguity.

\subsection{Nonlinear Young differential equation} 
\label{SUBSEC:YDE}

In this subsection, 
after stating the sewing lemma, 
we first provide a brief review on the construction  of nonlinear Young integrals
via the sewing lemma.
We then state a local well-posedness result
for a general nonlinear YDE (Proposition \ref{PROP:main}).

In \cite{Gub04}, the second author introduced 
the sewing lemma which plays a fundamental role
in (nonlinear) Young\,/\,rough integration theory;
see \cite[Chapter 4]{FH20} for a further discussion.
 The following version of the sewing lemma 
 is taken from~\cite{GT10} 
with slight modifications; 
see  
\cite[Proposition~2.3, Corollaries  2.4 and  2.5]{GT10}.
See also \cite[Lemma~4.2]{FH20}.

For $n\in \N$, we set $C^{1+}_{n,T}V = \bigcup_{\g>1}C^\g_{n,T}V$.

\begin{lemma}
\label{LEM:sew}

Let $V$ be a Banach space and 
fix $T>0$. 
Then,  there exists a unique linear map \textup{(}called the sewing map\textup{)}
$\Lambda:C^{1+}_{3,T} V \cap 
\Ker \updl|_{C_{3,T}V}
\to C^{1+}_{2,T}V$ such that 

\smallskip
\begin{enumerate}
\item[(i)] 
We have 
$\updl \Lambda h = h$
for each  $h\in C_{3,T}V\cap \Ker  \updl|_{C_{3,T} V}$.

\smallskip

\item[(ii)]
 For each $\z >1$,
the sewing map $\Lambda$ is continuous
from $C^\z _{3,T}V\cap 
\Ker  \updl|_{C_{3,T} V}$ to 
$C^\z _{2,T}V$ such that 
\begin{align*}
\|\Lambda h \|_{C^\z _{2,T}V} 
\le \frac{1}{2^\z - 2}  \| h \|_{C^\z _{3,T}V} 
\end{align*}

\noi
for any $h\in C^\z _{3,T}V$.

\smallskip
\item[(iii)] 
Given any  $g\in C_{2,T}V$ 
with 
$\updl g\in C^\z _{3,T}V$, 
there exists a unique
$f\in C([0, T];V)$  \textup{(}modulo an additive  constant\textup{)} 
such that 
$\updl f = (\Id - \Lambda \updl)g$. 
In addition, 
we have 
\begin{align*}
(\updl f)_{t,r} = \lim_{|\Pi([r,t])|\to 0} 
\sum_{j=0}^{n-1} g_{t_j,t_{j+1}}
\end{align*}

\noi
 for any $(t,r)\in \Dl_{2,T}$,
 where 
 the limit is over any partition
 $\Pi([r,t])$  
 of  $[r,t]$\textup{:} 
\[\Pi ([r,t]) = \{r = t_n < \dots < t_1 <  t_0 = t\}\]
whose mesh size 
$|\Pi([r,t])| = \sup_{j} |t_j-t_{j+1}|$ 
tends to $0$.

\end{enumerate}

\end{lemma}

Let us now provide a brief review
on the construction of a nonlinear Young integral 
associated with a nonlinear driver.
See
\cite[Section 3]{CGLLO}
for a full account  on the subject.
See also 
 \cite{HK, G23} 
 for the construction of  nonlinear Young integrals
 in a more general setting.

Given $T > 0$, 
let $X\in C^\g_{2, T} \Lip_k(V)$, 
where $C^\g_{2, T} \Lip_k(V)$
is 
as in \eqref{Ho3}, such that 
\begin{align}
X(0)=0
\qquad \text{and}\qquad 
\updl  X=0.
\label{X0}
\end{align}

\noi
In the following, we assume  that the driver $X_{t, r}$
is given by an integral operator
over the time interval $[r, t]$;
see, for example, \eqref{Xld1}
for the bilinear driver associated with the scaled modulated KdV \eqref{kdv4}.

Given  $\uu\in \mathcal{C}^{\alpha}([0,T];V)$ for some $0<\alpha<1$, 
our goal is to define 
 the nonlinear Young integral $\I^{X}(\uu)$ as the unique function 
 (in a suitable class of functions defined on $[0, T]$)
with $\I^X(\uu)(0) = 0$
 whose increment is given by 
\begin{align}
\label{Y1}
\big(\updl \I^{X}(\uu)\big)_{t,r}=X_{t,r}(\uu_{\bullet}).
\end{align}

\noi
Here, $\uu_\bul$ denotes 
the function $\uu$ evaluated at the variable of integration~$\bul$, satisfying $r \le \bul \le  t$.
As it is, the expression on 
the right-hand side of \eqref{Y1} is not well defined. 
By replacing $\bul$ by the left endpoint $r$, 
we have 
\begin{align}
X_{t, r}(\uu_\bul)
= X_{t, r}(\uu_r) + R_{t, r}
\label{Y2}
\end{align}

\noi
for some two-parameter process $R = R^{X, \uu}$.
Our goal is to find one such error term $R$ with sufficient regularity, 
which will allow us to define the nonlinear Young integral
 $\I^X(\uu)$
 as the unique limit of 
Riemann-Stieltjes type sums; see \eqref{YQ6}.

Define $\Theta$ on $\Dl_{2, T}$ by setting 
\begin{align*}
    \Theta_{t, r} = X_{t, r} (\uu_r), \quad (t, r) \in \Dl_{2, T}. 
\end{align*}

\noi
By applying the coboundary operator $\updl$ to \eqref{Y2} 
with \eqref{X0}, we see that  any error term $R$ (if it exists) must satisfy 
\begin{equation*}
(\updl R)_{t_1,t_2,t_3}
= - (\updl \Theta)_{t_1, t_2, t_3}
= 
X_{t_1,t_2}(\uu_{t_2})-X_{t_1,t_2}(\uu_{t_3}).
\end{equation*}

\noi
It follows  from  the regularity assumptions on $X$
and $\uu$
with~\eqref{Ho2} and~\eqref{Ho3}
that 
$\updl R 
 \in C^{\g + \al}_{3,T}V$; see \cite[Lemma 3.4\,(i)]{CGLLO}.
Thus, if $\g + \al > 1$, 
then we can apply the sewing lemma (Lemma~\ref{LEM:sew})
to define an error term $R$ by the relation:
\begin{align}
R = - \Ld \updl  \Theta
\in C^{\g + \al}_{2,T}V,
\label{YQ4} 
\end{align}
where $\Ld$ denotes the sewing map.
Then, 
putting
 \eqref{Y1},  
 \eqref{Y2},  
and \eqref{YQ4} together,   
we define the {\it nonlinear Young integral} $\I^X(\uu)$ of $\uu$
(with respect to the nonlinear Young driver $X$)
to be  a unique function 
in $ C^\g([0,T]; V)$
with $\I^X(\uu)(0) = 0$
whose increment is given by 
\begin{align*}
\updl (\I^X(\uu))
= (\Id - \Ld \updl )\Ta.
\end{align*}

\noi
In view of \eqref{YQ4} with  $\g + \al > 1$, 
we have 
\begin{align*}
 \lim_{|\Pi([0,t])|\to 0} 
\sum^{n-1}_{j=0} 
R_{t_j,t_{j+1}} = 
-  \lim_{|\Pi([0,t])|\to 0} 
\sum^{n-1}_{j=0} 
(\Ld \updl \Ta)_{t_j,t_{j+1}} = 
0.
\end{align*}

\noi
Hence, 
the nonlinear Young integral $\I^X(\uu)$ 
is given by 
the unique limit of 
Riemann-Stieltjes type sums:
\begin{align}
\I^X(\uu)(t) 
= \lim_{|\Pi([0,t])|\to 0} 
\sum^{n-1}_{j=0} \Theta_{t_j,t_{j+1}}
= \lim_{|\Pi([0,t])|\to 0} 
\sum^{n-1}_{j=0} X_{t_j,t_{j+1}}(\uu_{t_{j+1}}), 
\label{YQ6}
\end{align}

\noi
 where 
the limit is in the sense of Lemma \ref{LEM:sew}\,(iii).
See \cite[Lemma 3.4]{CGLLO}
for a summary of the basic properties
of nonlinear Young integrals.

\medskip

Next, we recall a local well-posedness result
for the  following nonlinear YDE:
\begin{equation}
\label{YDE1}
\uu = u_0 + \I^X(\uu),
\end{equation}

\noi
where $X$ is a nonlinear driver belonging to the class $\cX^{s,\g}_k([0, T] \times \T_\ld)$ defined in \eqref{X1}.
See \cite[Proposition 3.8]{CGLLO} for the proof
(when $\ld = 1$)
and other related results such as persistence of regularity, etc.

\begin{proposition}
\label{PROP:main}
Given   $s \in \R$, $\frac 12 < \g < 1$,  $k \in \N$,  $T \ge 1$, 
and $\ld \ge 1$, 
let     $X \in \cX^{s, \g}_k([0, T]\times \T_\ld)$, 
where $ \cX^{s, \g}_k([0, T]\times \T_\ld)$ is  defined  in \eqref{X1}.
Then, 
 the nonlinear YDE \eqref{YDE1} with 
 the driver $X$ 
is locally well-posed in $H^s(\T_\ld)$.
More precisely, given  $u_0\in H^s(\T_\ld)$, 
there exist $C_0 >0$ and   $\ta>0$, independent of $u_0$ and $X$, 
and 
 a unique solution $\uu \in \cC^\g([0, \tau]; H^s(\T_\ld))$
   to~\eqref{YDE1}  
with $\uu|_{t= 0} = u_0$, where the local existence time $\tau = 
\tau\big(\|u_0\|_{H^s(\T_\ld)}, \|X\|_{\cX^{s, \g}_k([0, T] \times \T_\ld)}\big) > 0$
satisfies 
\begin{align}
\tau \ge  C_0 
\Big(\|X\|_{\cX^{s,\g}_k([0, T] \times \T_\ld)}  (1+ \|u_0\|_{H^s(\T_\ld)})\Big)^{-\ta} .
\label{YD1}
\end{align}

\noi 
Moreover, given $0 < \al < \g$ with $\al + \gamma > 1$, 
there exists $C_\al > 0$ such that 
\begin{align*}
\| \uu(t) \|_{\CC^\al_\tau H^s_x}\le C_\al \|u_0\|_{H^s(\T_\ld)}.
\end{align*}

\end{proposition}

\begin{remark}\label{REM:theta}\rm
(i) 
Recall that  the proof of Proposition \ref{PROP:main}
(= \cite[Proposition 3.8]{CGLLO})
is based on a standard contraction argument
for the map $\G = \G^{X, u_0}$ on 
$\cC^\al([0, \tau]; H^s(\T_\ld))$, given by 
\begin{align*}
\G(\uu) = u_0 + \I^X(\uu), 
\end{align*}

\noi
where  $\alpha \in (0, \g)$ satisfies  $\g+\al>1$. 
Given $R > 0$, 
let 
$B_R\subset 
\cC^\al([0, \tau]; H^s(\T_\ld))$
denote the closed ball of radius $R$ centered at the origin.
It was shown in  \cite[(3.53) and (3.63)]{CGLLO}
that,  for any $\uu, \vv \in B_R$, 
 we have
\begin{align}
\begin{split}
\|\G(\uu)\|_{\CC^\al_\tau H^s_x}
&\le \|u_0\|_{H^s(\T_\ld)} \\
& \quad \, +  C\tau^{\g - \al}
\|X\|_{\cX^{s, \g}_k([0, T] \times \T_\ld)}
(1 + R)^{k-1}
\|\uu\|_{\CC^\al_{\tau } H^s_x}, \\ 
\|\G(\uu)- \G(\vv)\|_{\cC^\al _{\tau} H^s_x}
&\le
C\tau^{\g - \al}\|X\|_{\cX^{s, \g}_k([0, T] \times \T_\ld)}  (1+R)^{k-1}\|\uu- \vv \|_{\CC^\al _\tau H^s_x}.
\end{split}
\label{YD4}
\end{align}

\noi
Then, by setting $R = 2\|u_0\|_{H^s(\T_\ld)}$
and choosing $\tau > 0$
sufficiently small, 
we see that $\G$ is a contraction on the ball $B_R
\subset \cC^\al([0, \tau]; H^s(\T_\ld))$.
We point out that,  in view of \eqref{YD4},  
 the exponent $\theta$ in \eqref{YD1} 
 can be taken as  $\theta=\frac{1}{\gamma-\alpha}$.  

\medskip 
\noi
(ii) As pointed out in \cite[Remark 3.9\,(i)]{CGLLO}, 
a solution to the nonlinear YDE \eqref{YDE1} constructed in  Proposition \ref{PROP:main} 
satisfies  the following blowup alternative:
\begin{align} 
\label{BA1}
    T_{\ast} = \infty \qquad \text{or} \qquad \lim_{t \nearrow T_{\ast}} \| \uu (t) \|_{H^{s}} = \infty, 
\end{align}

\noi
where  $T_{\ast} \in (0, \infty]$ denotes the maximal time of existence.

\end{remark}

\begin{remark}\label{REM:sol1}\rm

Let $T > 0$.
Given a  $(\rho,\g)$-irregular function $w$ on $[0, T]$ in the sense of Definition~\ref{DEF:ir}, 
we say that $u$ is a solution to the modulated KdV  \eqref{kdv1} in $[0,\tau]\times \T$ 
(for some $0 < \tau \le T$)
with initial data $u |_{t = 0} = u_0 \in H^s(\T)$, 
if 
its modulated interaction representation $\uu$
belongs to $\CC^{\al} ([0, \tau]; H^s(\T))$ 
for  some $\al \in (1-\g, \g)$
and
satisfies the nonlinear YDE \eqref{YDE0}
on the time interval $[0, \tau]$, 
where 
$\I^{X} (\uu) $ is interpreted as the nonlinear Young integral of $\uu$
with respect to the bilinear driver 
$X$
defined in~\eqref{bd0} (see also \eqref{Xld1}).

We extend an analogous definition
to  the scaled modulated KdV \eqref{kdv4} on $\T_\ld$
and the scaled modulated $I$-KdV \eqref{kdv5}
on $\T_\ld$.

\end{remark}

\section{Refined global well-posedness}
\label{SEC:GWP}

In this section, we present a proof of Theorem \ref{THM:main}
by combining the $I$-method, the sewing lemma (Lemma \ref{LEM:sew}), 
and the scaling transforms \eqref{scaling1}
and \eqref{scaling2}.

\subsection{Scaled modulated KdV} 
\label{SUBSEC:I1}

In this subsection, we summarize the basic properties
of the scaling transforms
\eqref{scaling1} and \eqref{scaling2}.

Given  a modulation function $w$ on $[0, T]$ for some $T >0$, 
let $u$ be a solution to the modulated KdV \eqref{kdv1} on $[0, T] \times \T$
with initial data $u|_{t= 0} = u_0$.
Given $b \in \R$ and $\ld \ge 1$, 
let $u^\ld$ and $w^\ld$ be the scaled functions defined in 
\eqref{scaling1} and \eqref{scaling2}, respectively.
Then, as pointed out in Subsection \ref{SUBSEC:1.2}, 
the scaled function $u^\ld$ is a solution to the 
scaled modulated KdV \eqref{kdv4}
on $[0, \lambda^b T] \times \T_{\lambda}$
with the  scaled initial data 
$u^\ld|_{t = 0}  = u_0^\ld$, 
where $u_0^\ld$ is as in 
 \eqref{scaling4}.

Let us first record the scaling property of the 
$H^s(\T_\ld)$- and $\dot H^s(\T_\ld)$-norms defined in \eqref{FT4}.
When $\ld = 1$, 
the $H^s(\T)$-norm and $\dot H^s(\T)$-norm
are equivalent in the current mean-zero setting (recall \eqref{mean0}).
When $\ld \gg1$, 
we still have 
$\|f \|_{H^s(\T_\ld)}\sim_\ld \|f\|_{\dot H^s(\T_\ld)}$
for a mean-zero function $f$ on $\T_\ld$
but the implicit constant diverges as $\ld \to \infty$.
As in \cite{CGLLO}, 
we will  work with the homogeneous Sobolev spaces
in the following, 
since they enjoy a better scaling property; see also  \cite[Remark 7.2]{CGLLO}.
From~\eqref{scaling4}, 
we have $\F_{\T_{\ld}}(u_0^{\ld}) (n) = \ld^{-b+2} \F_{\T} (u_0) (\ld n)$, $n \in \Z_{\ld}$,
which yields
\begin{align} 
\label{scaling3}
    \| u_0^{\lambda} \|_{\dot{H}^s (\T_{\lambda})} = \lambda^{-b + \frac 32 - s} \| u_0 \|_{\dot{H}^s (\T)} 
\end{align} 

\noi 
for any $s \in \R$, while, for the non-homogeneous Sobolev spaces, we have 
\begin{align*}
\|u_0^\ld\|_{ H^s(\T_\ld)}
\le  \ld^{-b+\frac 32 - s} \|u_0\|_{ H^s(\T)}
\end{align*}
for $s\le 0$.
In particular, 
given any $s < 0$, by choosing 
\begin{align}
b>\frac{3}{2}-s
\label{scaling5}
\end{align}

\noi
and $\ld \gg_{\|u_0\|_{H^s(\T)}}1$,   we can make
the $\dot H^s(\T_\ld)$-norm of the scaled initial data $u_0^\ld$
small.
We note that, when $b = 3$ corresponding to the classical
KdV scaling \eqref{scaling0}, 
the condition \eqref{scaling5} reduces
to the scaling subcritical regime $s > - \frac 32$.

Let $\uu^\ld$ be 
the modulated interaction representation of the scaled function $u^\ld$ in \eqref{scaling1}, 
given by 
\begin{align}
\uu^\ld(t) = U^{w^\ld}(t)^{-1} u^\ld(t).
\label{ME0}
\end{align}

\noi
Then, $\uu^\ld$ (formally) satisfies the following integral equation:
\begin{align*}
\uu^{\ld} (t)=u^\ld_0
+
 \int_0^t   U^{w^\ld}(t')^{-1}  \dx\big(  ( U^{w^\ld} (t') \uu^\ld(t'))^2   \big)dt' .
\end{align*}

\noi
The associated bilinear driver $X^\ld$ 
is then given by 
\begin{align}
X_{t,r}^{\ld} (f_1, f_2)
=\int_{r}^{t}
U^{w^\ld}(t')^{-1}
\dx \big( (U^{w^\ld}(t') f_1  )(
U^{w^\ld}(t') f_2)\big) dt'
\label{Xld1}
\end{align}

\noi
for $0 \le r < t \le \ld^b T$ and functions $f_1$ and $  f_2$ on $\T_\ld$. 
When two arguments agree, 
we simply set 
$X_{t, r} (f) = X_{t, r} (f, f)$.
Moreover, 
when $\ld = 1$, we set 
\begin{align}
X_{t, r} = X^{1}_{t, r}.
\label{bd0}
\end{align}

\noi
By taking the Fourier transform
(recall \eqref{FT3}),\footnote{Recall that, in \eqref{bd0a},  the terms corresponding to $n_1\!=n_2\!=\!0$ do not appear under the mean-zero assumption.}
 we have
\begin{align}
\label{bd0a} 
    \mathcal{F}_{\T_{\ld}} \big( X_{t,r}^{\ld} (f_1, f_2) \big) (n) 
= \frac{i n}{\ld} \sum_{\substack{n_1, n_2 \in \Z_{\ld}^{\ast} \\n = n_1 + n_2}} \Phi^{w^{\ld}}_{t,r} (\Xi_{{\KDV}} (\bar{n})) \ft{f}_{1} (n_1) \ft{f}_{2} (n_2), 
\end{align} 

\noi 
where $\Phi_{t, r}^{w^{\ld}}$ is as in \eqref{rho2} with $w$ replaced by $w^{\ld}$ 
and $\Xi_{\KDV} (\bar{n})$  denotes the  resonance function
 given by 
\begin{align}
\Xi_\KDV (\bar n) &  = \Xi_\KDV (n,n_1,n_2) 
= - n^3+ n_1^3 +n_2^3.
\label{K3}
\end{align}

\noi
Recall that, under $n = n_1+ n_2$, we have
\begin{align}
\Xi_\KDV (\bar n) = - 3n n_1 n_2.
\label{K3a}
\end{align}

The following lemma states the  equivalence of the unscaled and scaled problems. 

\begin{lemma}
\label{LEM:tri3}
Let $b \in \R$ and $\ld \ge 1$.
Given $\rho \ge\frac 12$,  $\frac 12 < \g < 1$,  and $T> 0$, 
let  $w$ be $(\rho,\g)$-irregular on $[0, T]$. 
Fix $s \in \R$ satisfying \eqref{reg1}
and a mean-zero function $u_0 \in H^s(\T)$.
Then,  $u$ is a solution to the modulated KdV \eqref{kdv1}
on $[0, \tau ]\times \T$
with  $u|_{t = 0} = u_0$ 
for some $0 < \tau \le T$
if and only if the scaled function $u^\ld$ defined in \eqref{scaling1}
is a   solution to the scaled modulated KdV~\eqref{kdv4} on $[0, \ld^{b}\tau ]\times \T_\ld$
with the scaled initial data 
$u^\ld|_{t = 0} = u_0^{\ld}$ defined in~\eqref{scaling4}.
\end{lemma}

See Remark \ref{REM:sol1}
for the notion of a solution.
Since Lemma \ref{LEM:tri3} follows
from a straightforward modification of the proof of \cite[Lemma 7.3]{CGLLO}, 
we omit details.

\medskip

Next, we study the mapping property of the bilinear driver $X^\ld$ in \eqref{Xld1}.
Once we establish the following lemma,  
 local well-posedness of the scaled modulated KdV~\eqref{kdv4}
in $\dot H^s(\T_\ld)$ (which reduces to Theorem \ref{THM:old}\,(i) when $\ld = 1$)
follows from Proposition \ref{PROP:main}.

\begin{lemma} 
\label{LEM:tri1} 
Let $b \in \R$ and $\lambda \geq 1$. Given $\rho \geq \frac 12$, $\frac 12 < \gamma < 1$,  and $T > 0$,  let $w$ be $(\rho, \gamma)$-irregular on $[0, T]$.
 Then, for any $\rho \geq \frac 12$ and $s \in \R$ satisfying~\eqref{reg1}, 
  the bilinear driver $X^{\lambda}$
in~\eqref{Xld1}
 belongs to  $\dot \cX_2^{s, \gamma} ([0, \lambda^b T] \times \T_{\lambda})$ defined in \eqref{X1a},  satisfying the bound{\rm :}
\begin{align} 
\| X^{\lambda}_{t, r} \|_{\mathcal{L}_2 (\dot{H}^s (\T_{\lambda}))} 
\les \lambda^{-\frac 32 + b(1-\gamma) + s } \| \Phi^w \|_{\mathcal{W}^{\rho, \gamma}_{T}} |t - r|^{\gamma} 
\label{bd1}
\end{align} 

\noi 
for any $0 \leq r < t \leq \lambda^b T$.  
\end{lemma}

While Lemma \ref{LEM:tri1} follows from a straightforward modification of the 
proof of \cite[Lemma~7.1]{CGLLO} using the new scaling \eqref{scaling2}
(see, in particular, \eqref{MF2}),  
we present its proof since 
a similar modification is required
to obtain 
Lemma \ref{LEM:tri2} and Proposition \ref{PROP:com}
from \cite[Lemma~7.4]{CGLLO} and \cite[Proposition~7.6]{CGLLO}, respectively.

\begin{proof} 
In view of 
\cite[Lemma~3.1\,(i) and Remark~3.2\,(ii)]{CGLLO}, 
it suffices to prove the bound~\eqref{bd1}.

Arguing as in \cite[(7.15)]{CGLLO} 
with  \eqref{FT4} and  \eqref{bd0a},  we have 
\begin{align}
\label{MF1} 
\begin{split}
\| X^{\lambda}_{t, r} \|_{\mathcal{L}_2 (\dot{H}^s (\T_{\lambda}))} 
 \le\ld^{-\frac 12} \bigg( \sup_{n_1 \in \Z_{\ld}^{\ast}} \sum_{\substack{n, n_2 \in \Z_{\ld}^{\ast} \\ n_1 = n - n_2}}   |n|^{2s + 2} \frac {|\Phi^{w^{\ld}}_{t, r} (\Xi_{{\KDV}} (\bar{n}))|^2}{|n_1|^{2s} |n_2|^{2s}}\bigg)^\frac 12 , 
\end{split}
\end{align} 

\noi 
where $\Phi^{w^{\ld}}_{t, r}$ and $\Xi_{{\KDV}} (\bar{n})$ are as in \eqref{rho2} and \eqref{K3}, respectively.
 From \eqref{rho2}, \eqref{scaling2}, 
 a change of variables, 
 \eqref{rho1}, and \eqref{K3a}, we have 
\begin{align} 
\label{MF2} 
\begin{split}
    \big| \Phi_{t, r}^{w^{\lambda}} (\Xi_{{\KDV}} (\bar{n})) \big| 
    & = \bigg| \int_{r}^{t} e^{i \Xi_{{\KDV}} (\bar{n}) w^{\lambda} (t')} dt' \bigg| 
    = \ld^b \bigg| \int_{\ld^{-b} r}^{\ld^{-b} t} e^{i \ld^3 \Xi_{{\KDV}} (\bar{n}) w(t')} dt' \bigg|\\
    &\les \lambda^{b (1 - \gamma) - 3 \rho} \| \Phi^w \|_{\mathcal{W}_{T}^{\rho, \gamma}} |t - r|^{\gamma} |n n_1 n_2|^{-\rho} 
\end{split}
\end{align} 

\noi 
for any $n$, $n_1$, $n_2 \in \Z^{\ast}_{\lambda}$, satisfying $n = n_1 + n_2$.
Here, we used the fact that $\jb{\ld^3 n n_1 n_2} \sim 
\ld^3 |n n_1 n_2|$.

Given $n,\, n_1,\, n_2 \in \Z_{\ld}^{\ast}$, let $m = \ld n$ and $m_j = \ld n_j$, $j =1, 2$. 
Then it follows from \eqref{MF1}, \eqref{MF2}, 
and arguing as in the proof of \cite[Proposition 4.1\,(i)]{CGLLO} that 
\begin{align*}
    \| X_{t, r}^{\ld} \|_{\mathcal{L}_2 (\dot{H}^s (\T_{\ld}))} & \les \ld^{- \frac 12 + b(1 - \gamma) - 3\rho} \| \Phi^w \|_{\mathcal{W}_{T}^{\rho, \gamma}} |t - r|^{\gamma} \\ 
&\quad\, \times \bigg( \sup_{n_1 \in \Z_{\ld}^{\ast}} \sum_{\substack{n,n_2 \in \Z_{\ld}^{\ast} \\ n_1 = n - n_2}}  |n|^{2 - 4\rho} \frac {|n|^{2s + 2\rho}}{|n_1|^{2s + 2\rho} |n_2|^{2s + 2\rho}} \bigg)^{\frac 12} \\ 
    &\les \ld^{- \frac 12 + b(1 - \gamma) - 3\rho} \| \Phi^w \|_{\mathcal{W}_{T}^{\rho, \gamma}} |t - r|^{\gamma} \\ 
&\quad\,  \times 
\ld^{-1 + s + 3\rho}
\bigg( \sup_{m_1 \in \Z^{\ast}} 
\sum_{\substack{m,m_2 \in \Z^{\ast} \\ m_1 = m - m_2}}  |m|^{2 - 4\rho} \frac {|m|^{2s + 2\rho}}{|m_1|^{2s + 2\rho} |m_2|^{2s + 2\rho}} \bigg)^{\frac 12} \\ 
&\les \ld^{-\frac 32 + b(1 - \gamma) + s} \| \Phi^w \|_{\mathcal{W}_{T}^{\rho, \gamma}} |t - r|^{\gamma} 
\end{align*} 

\noi 
for any $0 \le r < t \le \ld^b T$, provided that \eqref{reg1} holds.
\end{proof}

\subsection{Scaled modulated $I$-KdV equation} \label{SUBSEC:I2}

In this and the next subsections, we carry out preliminary analysis
for the $I$-method.

Fix   $s<0$.
 Given $N \geq 1$,
 we define a smooth, even, 
 non-increasing (in $|\xi|$) 
 function 
 $m_{s, N} \in C^\infty(\R; [0, 1])$ 
 by setting
\begin{align*}
m_{s, N}(\xi)=
\begin{cases}
1, 
&\text{for }
|\xi|\leq N,   \\
\frac{ |\xi|^s }{ N^s}, 
&\text{for }
|\xi|\geq  2N.
\end{cases}
\end{align*}

\noi
Then, we define the so-called $I$-operator $I = I_{s, N}$
to be the Fourier multiplier operator with the multiplier $m_{s, N}$:
\begin{align}
\ft{I_{s, N}f}(\xi)=m_{s, N}(\xi)\ft{f}(\xi), 
\label{Ix2}
\end{align}

\noi
satisfying the  bound:
\begin{align}
\| f \|_{\dot H^s (\T_\ld) } \leq \ld^{-s}\| I f \|_{L^2(\T_\ld)}\quad\text{and}\quad
\| I f \|_{L^2(\T_\ld)}
\les N^{-s} \| f \|_{H^s(\T_\ld)}.
\label{I1}
\end{align}

By applying the $I$-operator defined in \eqref{Ix2} to 
the scaled modulated KdV \eqref{kdv4}, 
where $u^\ld$ and $w^\ld$ are as in~\eqref{scaling1}
and  \eqref{scaling2}, respectively, 
we obtain 
the following scaled modulated $I$-KdV system on the dilated circle $\T_{\lambda}$: 
\begin{equation}
\label{kdv5} 
\dt I u^{\lambda} + \partial_x^3 I u^{\lambda} \cdot \dt w^{\lambda} = \dx I \big( (u^{\lambda})^2 \big).
\end{equation} 

\noi 
By writing $\dx I \big( (u^{\lambda})^2\big) = \dx I \big( (I^{-1} I u^{\lambda}) (I^{-1} I u^{\lambda}) \big)$, 
we define  the associated bilinear driver $Y^{\lambda}$ 
by 
\begin{equation} 
Y^{\lambda} (f_1, f_2) = I X^{\lambda} (I^{-1} f_1, I^{-1} f_2)
\label{Xld2}
\end{equation} 

\noi 
for functions $f_1$ and $f_2$ on $\T_{\lambda}$, where $X^{\lambda}$  is  as in \eqref{Xld1}.

A straightforward modification of the proof of \cite[Lemma 7.4]{CGLLO}
with \eqref{Xld2}, \eqref{Xld1},  and \eqref{MF2}
yields the following lemma.
The bound \eqref{MIF2} follows from 
\eqref{MIF1}
and 
\cite[Lemma~3.1\,(i) and Remark~3.2\,(ii)]{CGLLO}.

\begin{lemma} 
\label{LEM:tri2}
Let 
$b \in \R$ and 
$\lambda \geq 1$. Given $\rho \geq \frac 12$, $\frac 12 < \gamma < 1$,  and $T > 0$, let $w$ be $(\rho, \gamma)$-irregular on $[0, T]$.
 Then, for any $s < 0$, satisfying \eqref{reg1}, 
 and $N \ge 1$, 
  the bilinear driver $Y^{\lambda}$ in~\eqref{Xld2}  satisfies
\begin{align} 
\label{MIF1}
    \| Y_{t, r}^{\lambda} \|_{\mathcal{L}_2 (L^2 (\T_{\lambda}))} \les \lambda^{-\frac 32 + b(1- \gamma)} \| \Phi^w \|_{\mathcal{W}^{\rho, \gamma}_{T}} |t - r|^{\gamma} 
\end{align} 

\noi 
for any $0 \leq r < t \leq \lambda^{b} T$. 
As a consequence, we have 
\begin{align} 
\label{MIF2}
    \| Y^{\lambda} \|_{\mathcal{X}_2^{0, \gamma} ([0, \lambda^b T] \times \T_{\lambda})}
\les \ld^{-\frac 32 + b(1- \g)} \| \Phi^{w} \|_{\mathcal{W}^{\rho, \gamma}_{T}}, 
\end{align}

\noi
where  $\mathcal{X}_2^{0, \gamma} ([0, \lambda^b T] \times \T_{\lambda})$ 
is as in \eqref{X1}.

\end{lemma}

As a corollary to Lemma
\ref{LEM:tri2}
and Proposition \ref{PROP:main}, 
we have the following local well-posedness in $L^2(\T_\ld)$ of the scaled modulated $I$-KdV 
\eqref{kdv5}.
See \cite[Corollary 7.5]{CGLLO}.

\begin{corollary}
\label{COR:I2}
Let 
$b \in \R$ and 
$\ld \ge 1$.
Given $\rho \ge\frac 12$,  $\frac 12 < \g < 1$,  and $T> 0$, 
let  $w$ be $(\rho,\g)$-irregular
on $[0, T]$. 
Fix  $s < 0$,  satisfying \eqref{reg1}, 
and $N \ge 1$.
Then, given any 
$v_0 \in L^2(\T_\ld)$ and  $t_0 \in [0, \ld^b T]$, 
there exist $C_0 > 0$, $\ta > 0$, both independent of $v_0$ and $t_0$, 
and a unique solution $Iu^{\ld}$
to 
the scaled modulated $I$-KdV equation~\eqref{kdv5} 
with $I u^\ld (t_0)= v_0$
on the time interval $[t_0, t_0 + \tau]\cap [0, \ld^b T]$, 
where the local existence time 
$\tau = \tau (\|v_0\|_{L^2 (\T_{\lambda})}, \|Y^{\lambda}\|_{\mathcal{X}_2^{0, \gamma} ([0, \lambda^b T] \times \T_{\lambda})}) > 0$ satisfies
\begin{align}
\tau  \ge C_0 \|Y^\ld\|_{\cX^{0,\g}_2([0, \ld^b T]\times \T_\ld)}^{-\ta} \big( 1 + \| v_0 \|_{L^2(\T_\ld)}\big)^{-\ta} .
\label{ME0a} 
\end{align}

\noi 
Moreover, given any $0 < \al < \gamma$ with  $\alpha + \gamma > 1$, 
there exists $C_\al > 0$ such that 
\begin{align*}
 \| I\uu^\ld \|_{\CC^\al ([t_0, t_0 + \tau] \cap [0, \ld^b T]; L^2(\T_\ld))}\le 
 C_\al
 \|v_0\|_{L^2(\T_\ld)}, 
\end{align*}

\noi
where $\uu^\ld$ denotes the modulated interaction representation of $u^\ld$
defined in  \eqref{ME0}.

\end{corollary}

\begin{remark}\rm \label{REM:tau}

Recall that, given a target time $T \gg1$, 
we choose $\ld = \ld(T) \gg1$ (which tends to $\infty$
as $T\to \infty$) in the application of the $I$-method.
When $b = 3$ considered in \cite{CGLLO}, 
the exponent of $\ld$ in \eqref{MIF2} is $\frac 32 - 3\g$, 
which is negative for $\frac  12 < \g < 1$.
Namely, there is no loss  in \eqref{MIF2} coming from $\ld \gg1$.

Our main goal in this paper is 
to show that, given $s < 0$ with $|s| \gg 1$
and $\frac 12 < \g < 1$, there exists $\rho \gg 1$
such that, given a $(\rho, \g)$-irregular modulation function $w$, 
the modulated KdV~\eqref{kdv1} on $\T$ is globally well-posed in $H^s(\T)$.
Then, in view of \eqref{scaling5}, 
we need to take $b \gg 1$ when $|s| \gg 1$.
For such $b \gg 1$, 
the exponent of $\ld$ in \eqref{MIF2} is 
positive since $\frac 12 < \g < 1$.
Namely, there is a loss of the positive power of $\ld$ in \eqref{MIF2}, 
which requires additional care as compared to the situation in \cite{CGLLO}.

For simplicity of the presentation, 
we only consider the case:
\begin{align}
 -    \frac {3}{2} + b (1 - \gamma) >  0,  \quad \text{namely,} \quad  b > \frac {3}{2 (1 - \g)}
\label{bg} 
\end{align}

\noi
in the proof of Theorem \ref{THM:main} presented in Subsection \ref{SUBSEC:I4}.
The loss $\ld^{-\frac 32 + b(1- \g)}$
in \eqref{MIF2} forces 
the local existence time $\tau$ of the scaled modulated $I$-KdV~\eqref{kdv5}
on $\T_\ld$ in Corollary \ref{COR:I2} 
to shrink to $0$
as $\ld \to \infty$.
More precisely, given $\al \in (1- \g, \g)$, 
it follows from \eqref{YD4} in 
Remark~\ref{REM:theta}\,(i)
that we can choose $\ta$ in \eqref{ME0a} as 
\begin{align}
\ta = \frac 1{\g- \al}.
\label{A1}
\end{align}

\noi
Then, 
from \eqref{MIF2} and \eqref{ME0a}, we see that, given $\ld \gg 1$
(with $\| v_0 \|_{L^2(\T_\ld)} = O(1)$),  
the local existence time $\tau$
in Corollary~\ref{COR:I2}
 is given by 
\begin{align}
\tau \sim \ld^{\left(\frac 32 - b(1- \g)\right)\frac 1{\g - \al}}\too  0, 
\label{A2}
\end{align}

\noi
as $\ld \to \infty$, 
thus requiring more iteration steps in order to reach the target time $\ld^b T$
for the scaled modulated $I$-KdV \eqref{kdv5}; see \eqref{iteration}.
Lastly, we point out that
in maximizing the local existence time $\tau$ in \eqref{A2}, 
we choose 
\begin{align}
\al = 1 - \g+ \eps
\label{A3}
\end{align}

\noi
for sufficiently small $\eps > 0$.

\end{remark}

\subsection{Commutator  estimate} \label{SUBSEC:I3}

Define the (scaled) modified energy: 
\begin{align} 
\label{ME1}
    \| I u^{\lambda} (t) \|^2_{L^{2} (\T_{\lambda})} =  \|I\uu^{\ld}(t)\|^2_{L^2(\T_{\ld})}, 
\end{align}  

\noi
where $\mathbf{u}^{\lambda}$ is the modulated interaction representation of $u^{\lambda}$ defined in \eqref{ME0}.
From the fact that $u(t)$ has spatial mean $0$, \eqref{scaling1}, 
\eqref{scaling3},  and \eqref{I1}, we have
\begin{align}
\begin{split}
\|u(t) \|_{H^s(\T)} 
& \sim \|u(t) \|_{\dot H^s(\T)}
  = 
\ld^{b-\frac 32 + s}\|u^\ld (\ld^{b} t) \|_{\dot H^s(\T_\ld)}\\
& \les
\ld^{b - \frac 32}
 \|I\uu^{\ld}(\ld^{b}t)\|_{L^2(\T_{\ld})}.
\end{split}
\label{BA2}
\end{align}

\noi
Fix a target time $T \gg 1$.
In order to guarantee existence
of a solution $u$ to the modulated KdV \eqref{kdv1}
on $[0, T] \times \T$, 
it follows from 
 the blowup alternative \eqref{BA1}
 and \eqref{BA2}
that  
it suffices to show that the modified energy 
$ \|I\uu^{\ld}(t)\|^2_{L^2(\T_{\ld})} $ in \eqref{ME1}
remains finite on the time interval $[0, \ld^b T]$.

Recall from \cite[(7.29)]{CGLLO} that 
\begin{align} 
\begin{split}
& \| I \uu^{\lambda} (t) \|^2_{L^2 (\T_{\lambda})} -  \| I \mathbf{u}^{\lambda} (r) \|^2_{L^2 (\T_{\lambda})} \\
& \quad =  2 \big\langle I \mathbf{u}^{\lambda} (r), \com_{t, r} (\mathbf{u}^{\lambda} (r), \mathbf{u}^{\lambda} (r)) \big\rangle_{L^2 (\T_{\lambda})} + R_{t, r}^{\lambda} 
\end{split}
\label{IDiff1} 
\end{align}  

\noi 
for any $t > r \geq 0$, where the commutator $\com_{t, r}$ and the remainder $R_{t, r}^{\lambda}$ are given by 
\begin{align}
\label{IDiff2}   
\begin{split}
    \com_{t, r} (f_1, f_2) & = I X^{\lambda}_{t, r} (f_1, f_2) - X^{\lambda}_{t, r} (I f_1, I f_2), \\ 
    R^{\lambda}_{t, r} & = \| I \mathbf{u}^{\lambda} (t) - I \mathbf{u}^{\lambda} (r) \|^2_{L^2 (\T_{\lambda})} \\ 
    & \quad + 2 \big\langle I \mathbf{u}^{\lambda} (r), I \mathbf{u}^{\lambda} (t) - I \mathbf{u}^{\lambda} (r) - I X^{\lambda}_{t, r} (\mathbf{u}^{\lambda} (r), \mathbf{u}^{\lambda} (r)) \big\rangle_{L^2 (\T_{\lambda})}. 
\end{split}
\end{align}

\noi
We now state
 a key commutator estimate,
extending
\cite[Proposition 7.6]{CGLLO}
(for $b = 3$) 
to the case of general $b \in \R$.
Its proof follows from a straightforward modification of
the proof of 
\cite[Proposition 7.6]{CGLLO}, 
using \eqref{MF2}, 
and thus we omit details.

\begin{proposition} 
\label{PROP:com}
Let $b \in \R$ and $\lambda \geq 1$. Given $\rho > \frac 12$, $\frac 12 < \gamma <1$, and $T > 0$, let $w$ be $(\rho, \gamma)$-irregular on $[0, T]$, and let $X^{\lambda}$ be as in \eqref{Xld1}. Given $s < 0$, satisfying \eqref{reg1}, and $N \ge 1$. Then, we have  
\begin{align} 
\label{com1}
\begin{split}
     \| \com_{t, r} (f_1, f_2) \|_{L^2 (\T_{\lambda})} 
& \les K_b (\lambda, N) \| \Phi^w \|_{\mathcal{W}^{\rho, \gamma}_{T}} |t - r|^{\gamma} \\
& \quad \times \| I f_1 \|_{L^2 (\T_{\lambda})} \| I f_2 \|_{L^2 (\T_{\lambda})}
\end{split}
\end{align}  

\noi 
for any $0 \leq r <t \leq \lambda^{b} T$, 
where $K_b ( \lambda, N)$ is given by 
\begin{align} 
\label{com2}
K_b (\lambda, N) = \begin{cases}
\lambda^{b - 3\rho - b \gamma} N^{\frac 32 - 3\rho}, & \textit{if } \frac 12 < \rho < \frac 32, \\ 
\lambda^{b - \frac 92 - b \gamma} N^{-3} (\log N + \log \lambda)^{\frac 12}, & \text{if } \rho = \frac 32,  \\ 
\lambda^{b - \frac 32 -2\rho - b \gamma } N^{-2 \rho}, & \text{if } \rho > \frac 32. 
\end{cases}
\end{align} 
\end{proposition}

As a corollary to
Proposition \ref{PROP:com} and the sewing lemma (Lemma \ref{LEM:sew}), 
we have the following bound on  the remainder term $R_{t, r}^{\ld}$ in  \eqref{IDiff2}; 
see \cite[(7.50)-(7.52)]{CGLLO}. 
Fix an interval $J \subset [0, \lambda^b T]$ and $0 < \alpha < \gamma$ with $\alpha + \gamma > 1$. Then, we have
\begin{align}
\label{R1} 
    \big| R^{\lambda}_{t_1, t_2} \big| 
\les K_b ( \lambda, N) \| \Phi^{w} \|_{\mathcal{W}_{T}^{\rho, \gamma}} |t_1 - t_2|^{\alpha + \gamma} \| I \mathbf{u}^{\lambda} \|^3_{\mathcal{C}^{\alpha} (J; L^2 (\T_{\lambda}))}
\end{align} 

\noi 
for any $t_1 > t_2 \geq 0$ belonging to the interval $J$, where $K_b ( \ld, N)$ is as in \eqref{com2}.

We note that  
the power of $N$ in \eqref{com2}
for 
$ K_b ( \ld, N)$ appearing in both \eqref{com1} and \eqref{R1}
is negative,\footnote{When $\rho = \frac 12$, this is not true 
and this is the reason that we have $\rho > \frac 12$ for global
well-posedness in Theorems~\ref{THM:old}\,(ii) and \ref{THM:main}, 
while we have $\rho \ge \frac 12$ for local well-posedness
in Theorem~\ref{THM:old}\,(i).} which 
allows us to establish
 almost conservation of the modified energy 
$\| I \uu^\ld(t)\|^2_{L^2(\T_\ld)}$
(by suitably choosing $\ld = \ld(N)\gg1$).

\subsection{Proof of Theorem \ref{THM:main}} \label{SUBSEC:I4}

We conclude this paper by presenting a proof of 
Theorem \ref{THM:main}.

Given $\rho > \frac 12$ and $\frac 12 < \gamma < 1$,  let $w$ be $(\rho, \gamma)$-irregular on $\R_{+}$. 
 Fix $u_0 \in H^s ({\T})$ for $s < 0$ satisfying \eqref{reg1} 
such that Theorem \ref{THM:old}\,(i) guarantees 
local existence of a solution $u$ to the modulated KdV \eqref{kdv1} on $\T$.

Fix a target time $T \gg 1$.
In view of the blowup alternative \eqref{BA1} and \eqref{BA2}, 
Theorem~\ref{THM:main} follows once we show that 
\begin{align*}
    \sup_{0 \leq t \leq \lambda^b T} \| I \mathbf{u}^{\lambda} (t) \|_{L^{2} (\T_{\lambda})} 
\leq C_{s, b} (T, \lambda) < \infty 
\end{align*} 

\noi 
 for a suitable choice of $b > \frac 32 - s$ (see \eqref{scaling5}), 
 $N = N(T) \gg1$,  and $\ld = \ld(T) \geq 1$.
For simplicity of notation, we often suppress dependence
on $s$ and $b$ (and also on $\rho$, $\g$, and $w$) in the following.

From \eqref{I1} and \eqref{scaling3}, we have 
\begin{align*} 
    \| I u_0^{\lambda} \|_{L^2 (\T_{\lambda})} \les \lambda^{-b + \frac 32 - s} N^{-s} \| u_0 \|_{H^s (\T)}. 
\end{align*} 

\noi 
Then, given small $\eps_0 > 0$, we can choose 
\begin{align} 
\label{IT2}
    \lambda \sim N^{- \frac {2s}{2b + 2s - 3}}, 
\end{align} 

\noi 
such that 
\begin{align} 
\label{IT3}
    \| I u_0^{\lambda} \|_{L^2 (\T_{\lambda})}^2 \leq \eps_0,
\end{align}  

\noi 
provided that $b>\frac{3}{2}-s$. Here, the implicit constant in \eqref{IT2} depends only on $\| u_0 \|_{H^s (\T)}$
and the exponent of $N$ in \eqref{IT2} is positive for $s < 0$ and $b > \frac 32 - s$.

Fix  $ \al \in (1- \g, \g)$
(to be determined later). 
Let $0 < \tau \le 1$  be the local existence time for the scaled modulated $I$-KdV \eqref{kdv5},  
 corresponding to initial data $v_0$ with $\| v_0 \|^2_{L^2 (\T_{\lambda})} =  2 \eps_0$;
see \eqref{A2}.
 More precisely, if we have 
\begin{align} 
\label{Ia} 
    \| I \mathbf{u}^{\lambda} (t_0) \|^2_{L^2 (\T_{\lambda})} \leq 2 \eps_0
\end{align} 

\noi 
for some $0 \leq t_0 \leq \lambda^b T$, then Corollary \ref{COR:I2} 
(see \eqref{A2}) guarantees that the solution $I u^{\lambda}$ to the scaled modulated $I$-KdV \eqref{kdv5} exists on the time interval $[t_0, t_0 + \tau] \cap [0, \lambda^b T]$, satisfying the bound: 
\begin{align}
\label{Ib} 
\| I \mathbf{u}^{\lambda} \|^2_{\mathcal{C}^{\alpha} ([t_0, t_0 + \tau] \cap [0, \lambda^b T]; L^2 (\T_{\lambda}))} \leq C_\al \eps_0. 
\end{align}  

\noi 
Note that  \eqref{IT3} states
that \eqref{Ia} is satisfied with $t_0 = 0$, 
and thus 
 the solution $I u^{\ld}$ exists on $[0, \tau]$.

As declared in 
Remark \ref{REM:tau}
(see \eqref{bg}),
we assume that 
\begin{align}
b > b_*(s, \g) : = \max\bigg(\frac {3}{2} - s, \, \frac {3}{2 (1 - \g)}\bigg)
\label{bg2}
\end{align}

\noi
in the following.
From \eqref{IDiff1},  \eqref{IT3}, Proposition \ref{PROP:com}, \eqref{R1}, and \eqref{Ib} with $t_0 = 0$, we have  
\begin{align} 
\label{In2}
\begin{split}
&\| I \mathbf{u}^{\lambda} (t) \|^2_{L^2 (\T_{\lambda})} \\
&\quad \le \| I \uu_0^{\lambda} \|_{L^2 (\T_{\lambda})}^2+2\| I \uu_0^{\ld} \|_{L^{2} (\T_{\ld})} \| \com_{t, 0} (\uu_0^{\ld}, \uu_0^{\ld}) \|_{L^2 (\T_{\ld})} + |R^{\ld}_{t,0}| \\
&\quad \le \eps_0+C 
\tau^\g K_b (\lambda, N)
\| \Phi^{w} \|_{\mathcal{W}_{T}^{\rho, \gamma}}
\end{split}
\end{align}  

\noi
for any $0 \le t \le \tau$, 
where $K_b (\lambda, N)$ is as in \eqref{com2}.

Suppose that 
by choosing $N \gg 1$ (and thus $\ld = \ld(N) \gg 1$ in view of \eqref{IT2}), 
we can make 
$\tau ^\g K_b (\lambda,  N)$ 
as small as we need, where 
 $\tau$ is as in \eqref{A2}.
 Then, from \eqref{In2}, we have
\begin{align*} 
\sup_{0 \leq t \leq \tau} \| I \mathbf{u}^{\lambda} (t) \|^2_{L^2 (\T_{\lambda})} 
 \leq 2 \eps_0. 
\end{align*} 

\noi 
 As a consequence, the solution $I u^{\ld}$ exists on $[\tau, 2 \tau]$ satisfying the bounds \eqref{Ia} and \eqref{Ib} with $t_0 = \tau$, which allows us to iterate this process.

 After $j$ steps, 
 we see that  the solution $I u^{\lambda}$ exists on $[0, j \tau]\cap [0, \ld^b T]$, satisfying \eqref{Ia} and \eqref{Ib} with $t_0 = (j-1) \tau$.
 Furthermore, we have 
\begin{align*}
    \sup_{0 \leq t \leq j \tau} \| I \mathbf{u}^{\lambda} (t) \|^2_{L^2 (\T_{\lambda})} 
& \leq  \eps_0 + j C \tau ^\g 
     K_b (\lambda, N) \| \Phi^w \|_{\mathcal{W}^{\rho, \gamma}_{T}}\\
& \le 2\eps_0, 
\end{align*}

\noi
where the last inequality holds 
for $j = 1, \dots, \big[\frac{\ld^b T}{\tau}\big]+1$, 
if 
\begin{align}
\label{iteration}
j \tau ^\g 
     K_b (\lambda, N)
\les \frac {\lambda^{b} T}{\tau}
\tau ^\g 
     K_b (\lambda, N)
\ll 1.
\end{align} 

\noi 
It follows from   \eqref{A2}
that the second inequality in \eqref{iteration} holds 
if we can choose  $N = N(T) \gg1 $ and $\ld = \ld(T) \gg1 $, 
satisfying \eqref{IT2},  such that 
\begin{align} 
\label{CON1}
\lambda^{b - (1-\g)\left( \frac 32 - b (1 - \gamma) \right) \frac 1{\g - \al}}  K_b( \lambda, N)
= \eps_1  
\end{align}

\noi
for some small $\eps_1 > 0$.

Fix  $\rho > \frac 12 $  
and $\frac 12 < \g < 1$.
Our task is 
to maximize the range of 
 $s < 0$
 which 
 satisfies~\eqref{reg1}
 (for local well-posedness) 
 and also 
 guarantees~\eqref{CON1} under \eqref{IT2}
 by making a  suitable choice
 of {\it independent} parameters 
 $\al \in (1 - \g, \g)$ and  
$b>b_*(s, \g)$, 
where $b_*(s, \g)$ is as in \eqref{bg2}.
In view of \eqref{bg}
and \eqref{com2}
(which says that $K_b(\ld, N)$ is independent of $\al$), 
we need to choose $\al$ which maximizes $\g - \al$.
Namely, we set $\al = 1 - \g + \eps$
for sufficiently small $\eps > 0$
as in~\eqref{A3}.

The following two lemmas show that, 
given $\rho > \frac 12 $,  
$\frac 12 < \g < 1$,  and $s < 0$, 
satisfying the hypothesis of Theorem \ref{THM:main}, 
by choosing
 $\al = 1 - \g + \eps$
for sufficiently small $\eps > 0$
as in~\eqref{A3}, 
 there exists
$b>b_*(s, \g)$, 
where $b_*(s, \g)$ is as in \eqref{bg2}, 
such that \eqref{CON1} 
under \eqref{IT2}
holds for any sufficiently large $N\gg 1$. 
This proves 
 Theorem \ref{THM:main}.

The first  lemma treats the case $\rho \ge \frac 32$.

\begin{lemma}\label{LEM:COND1}
Suppose that  $\rho\geq\frac{3}{2}$, $\frac{1}{2}<\gamma<1$, 
and  $s<0$ satisfy one of the following\textup{:} 
\begin{equation} 
\begin{split}
\textup{(i)} &\ \ \tfrac{1}{2}<\gamma <  \tfrac{\sqrt{5}-1}{2}, \quad  \tfrac {3}{2} \le \rho < \tfrac {3 (1 - \gamma)^2}{2 (-\gamma^2 - \gamma +1)}, \quad \text{and} \quad s \geq - \rho, \\ 
\textup{(ii)} &\ \ \tfrac{1}{2}<\gamma <  \tfrac{\sqrt{5}-1}{2}, \quad \rho \geq \tfrac{3 (1 - \gamma)^2}{2 (- \gamma^2 - \gamma + 1)}, \quad \text{and} \quad s > - \tfrac{4\rho (2 \gamma - 1) + 3 (1 - \gamma)^2}{2 (-\gamma^2 + 3 \gamma - 1)}, \\
\textup{(iii)} &\ \ \tfrac{\sqrt{5}-1}{2} \le \gamma < 1, \quad \rho \geq \tfrac 32, \quad \text{and} \quad s \ge -\rho.
\end{split}
\label{cond1}
\end{equation}

\noi
Then, $\rho$, $\g$,  and $s$ satisfy \eqref{reg1}.
Moreover, 
by choosing
 $\al = 1 - \g + \eps$
for sufficiently small $\eps > 0$
as in~\eqref{A3}, 
 there exists
$b>b_*(s, \g)$, where $b_*(s, \g)$ is as in \eqref{bg2},  such that \eqref{CON1} 
under \eqref{IT2}
holds for any sufficiently large $N\gg 1$. 

\end{lemma}

The second lemma treats the case $\frac 12 < \rho < \frac 32$.

\begin{lemma} 
\label{LEM:COND2}
Suppose that  $\frac{1}{2} < \rho<\frac{3}{2}$,  $\frac{1}{2}<\gamma<1$, 
and  $s<0$ satisfy one of the following\textup{:} 
\begin{align} 
\label{cond2}
\begin{split}
\textup{(i)} &\ \ \tfrac{1}{2} < \rho \le \tfrac{3 \g}{6 \g - 2} \quad \text{and} \quad s > \big(\tfrac{3}{2}-3\rho\big) \g, \\ 
\textup{(ii)} &\ \  \tfrac{3 \g}{6 \g - 2} < \rho < \tfrac {3}{2} \quad \text{and} \quad s \ge - \rho. 
\end{split}
\end{align} 

\noi
Then, $\rho$, $\g$,  and $s$ satisfy \eqref{reg1}.
Moreover,  
by choosing
 $\al = 1 - \g + \eps$
for sufficiently small $\eps > 0$
as in~\eqref{A3}, 
there exists
$b>b_*(s, \g)$, where $b_*(s, \g)$ is as in \eqref{bg2},  such that \eqref{CON1} 
under \eqref{IT2}
holds for any sufficiently large $N\gg 1$.

\end{lemma}

It remains to prove 
 Lemmas \ref{LEM:COND1} and \ref{LEM:COND2}.
We first present a proof of
 Lemma \ref{LEM:COND1}.

\begin{proof}[Proof of Lemma \ref{LEM:COND1}]
We only consider  the case $\rho > \frac 32$, since the case $\rho = \frac 32$ essentially follows from  the same argument.

For simplicity of notation, 
let 
\begin{align}
\ta = \frac 1{\g - \al} = \frac 1 {2\g - 1 - \eps}
\label{A4}
\end{align}

\noi
for sufficiently small $\eps > 0$ (to be chosen later)
as in \eqref{A1}, where the second equality follows from \eqref{A3}.
Then, 
it follows from 
 \eqref{com2} and \eqref{IT2}
 that  \eqref{CON1} holds if 
\begin{align}
1 \ll N^{-\frac{2s}{2b+2s-3} \left(- b(2-\g + (1 - \g)^2 \theta) + 2\rho + \frac{3}{2} + \frac 32 (1 - \gamma) \theta \right)+2\rho}
\label{nn0}
\end{align}

\noi
for $N \gg1 $, 
provided that 
 the exponent of $N$ in \eqref{nn0} is positive, which is equivalent to 
\begin{align} 
\label{COND1a}
2b(s( 2-\gamma +(1-\gamma)^2\theta)+2\rho)>6\rho+3s(1+(1-\gamma)\theta).
\end{align}

We define $A$, $B$, and $S$ by setting
\begin{align}
\begin{split}
    A &= A (s, \rho, \g, \theta) := s (2 - \gamma + (1 - \gamma)^2 \theta) + 2\rho, \\ 
    B &= B (s, \rho, \g, \theta) := 6 \rho + 3s(1 + (1 - \gamma) \theta), \\ 
    S &= S (\rho, \g, \theta) := -  \frac {2 \rho}{2 - \g + (1 - \g)^2 \theta}
\end{split} 
\label{nn1}
\end{align} 

\noi
such that the condition \eqref{COND1a}
is written as 
\begin{align}
2Ab > B.
\label{A5}
\end{align}

\smallskip
\noi
$\bul$ {\bf Case 1:} 
We first consider the case $A > 0$, 
which is equivalent to $s > S$ 
in view of \eqref{nn1}.
In this case, from \eqref{A5}, 
we see that 
\eqref{CON1} can be  satisfied 
by choosing  
 \[b > \max\bigg(b_*(s, \g), \,  \frac {B}{2A}\bigg).\]

\smallskip
\noi
$\bul$ {\bf Case 2:} 
 Next, we consider the case $A \le 0$.
In this case, a direct computation with \eqref{nn1}, $ \g < 1$, 
and \eqref{A4}
shows that $B<0$.
Suppose that $A = 0$, namely $s = S$.
In this case, \eqref{A5} (and hence \eqref{CON1})
is satisfied by choosing 
$b > b_*(s, \g)$.

It remains to consider the case $A < 0$, namely $s < S$.
In this case, 
it follows from  \eqref{A5}
that 
 \eqref{CON1} holds by choosing  $b$ such that 
\begin{align*}
    b_*(s, \g) < b < \frac {B}{2A}, 
\end{align*} 

\noi 
provided that 
\begin{align}
    b_*(s, \g) <  \frac {B}{2A}.
\label{nn2}
\end{align}

\noi
From 
\eqref{bg2} and \eqref{nn1}, 
the condition \eqref{nn2} is equivalent to 
\begin{align}
    s >\max( - 2 \rho \gamma, \, G), 
\label{nn3}
\end{align}

\noi 
where $G =   G(\rho, \g, \ta)$ is given by 
\begin{align}
 G = G(\rho, \g, \ta) := \frac {- 4\rho + 3 (1 - \g)(1 - \g \theta)}{2 (2 - \g + (1 - \g)^2 \theta)}
 \label{Z1}
\end{align}

\noi
with  $\ta$ as in \eqref{A4}.
From \eqref{nn1} with \eqref{A4} and $\frac 12 < \g < 1$, 
we note that 
 the right-hand side of~\eqref{nn3} is  strictly less than $S$.

\medskip

Recall from \eqref{reg1} that 
we need to impose $s \ge - \rho$ when $\rho \ge \frac 32$.
Noting that  $- \rho > - 2 \rho \g$
for $\g > \frac 12$, 
it follows from Cases 1 and 2 that 
Lemma \ref{LEM:COND1} holds
if 
\begin{align*}
s\ge - \rho \qquad\text{and}\qquad s > G.
\end{align*}

\noi
A direct computation with \eqref{Z1} and \eqref{A4} shows that 
\begin{align}
\text{$- \rho > G$ \ \ is equivalent to }\ \ 
2\rho \big(g(\g) -\g \eps\big) > - 3 (1-\g) (1-\g+\eps), 
\label{A6}
\end{align}

\noi
where $g(\g) = \g^2 +\g - 1$.
When $g(\g) \ge  0$, 
namely, 
$\frac{\sqrt{5}-1}{2} \le \g < 1$, 
\eqref{A6} holds true by choosing sufficiently small $\eps > 0$.
This yields (iii) in \eqref{cond1}.

When $g(\g) <  0$, 
namely, 
$\frac 12 < \g <  \frac{\sqrt{5}-1}{2}$, 
(i) and (ii) in \eqref{cond1} follow from 
examining 
when~\eqref{A6} holds or not
(by  choosing sufficiently small $\eps > 0$).
\end{proof}

Finally, we present a proof of
 Lemma \ref{LEM:COND2}.

\begin{proof}[Proof of Lemma \ref{LEM:COND2}]
We closely follow the proof of Lemma \ref{LEM:COND1}.

Let $\ta$ be as in \eqref{A4}
for sufficiently small $\eps > 0$
 (to be chosen later).
Then, 
it follows from~\eqref{com2} and \eqref{IT2}
 that  \eqref{CON1} holds if 
\begin{equation}
1\ll N^{-\frac{2s}{2b+2s-3}\left(-b(2-\gamma + (1 - \gamma)^{2} \theta ) + 3 \rho 
+ \frac{3}{2} (1 - \gamma) \ta \right)+3\rho-\frac{3}{2}}
\label{A8}
\end{equation}

\noi
for $N \gg1 $, 
provided that 
 the exponent of $N$ in \eqref{A8} is positive, which is equivalent to 
\begin{equation} 
\label{COND2a}
2b\Big(s(2-\gamma+(1-\gamma)^{2}\theta)+3\rho-\frac{3}{2}\Big)> 3 s (1 + (1 - \gamma) \theta) + 9\rho-\frac{9}{2} .
\end{equation}

We define $\wt A$, $\wt B$, and $\wt S$ by setting
\begin{align} 
\begin{split}
    \wt A &= \wt A (s, \rho, \g, \theta) = s (2 - \g + (1 - \g)^2 \theta) + 3 \rho - \frac 32, \\ 
\wt  B &= \wt B (s, \rho, \g, \theta) = 3 s (1 + (1 - \g) \theta) + 9 \rho - \frac {9}{2}, \\ 
\wt S &= \wt S (\rho, \g, \theta) = 
 \frac {3 - 6\rho}{2(2 - \g + (1 - \g)^2 \theta)}
\end{split} 
\label{nnn1}
\end{align}

\noi
such that the condition \eqref{COND2a}
is written as 
\begin{align}
2\wt Ab > \wt B.
\label{A9}
\end{align}

\noi 
\smallskip
\noi
$\bul$ {\bf Case 1:} 
We first consider the case $\wt A > 0$, 
which is equivalent to $s > \wt S$ 
in view of \eqref{nnn1}.
In this case, from \eqref{A9}, 
we see that 
\eqref{CON1} can be  satisfied 
by choosing  
 \[b > \max\bigg(b_*(s, \g), \,  \frac {\wt B}{2 \wt A}\bigg).\]

\smallskip
\noi
$\bul$ {\bf Case 2:} 
 Next, we consider the case $\wt A \le 0$.
In this case, a direct computation with \eqref{nnn1}, $ \g < 1$, 
and \eqref{A4}
shows that $\wt B<0$.
Suppose that $\wt A = 0$, namely $s = \wt S$.
In this case, \eqref{A9} (and hence \eqref{CON1})
is satisfied by choosing 
$b > b_*(s, \g)$.

It remains to consider the case $\wt A < 0$, namely $s < \wt S$.
In this case, 
it follows from  \eqref{A9}
that 
 \eqref{CON1} holds by choosing  $b$ such that 
\begin{align*}
    b_*(s, \g) < b < \frac {\wt B}{2\wt A}, 
\end{align*} 

\noi 
provided that 
\begin{align}
    b_*(s, \g) <  \frac {\wt B}{2\wt A}.
\label{Xnn2}
\end{align}

\noi
From 
\eqref{bg2} and \eqref{nnn1}, 
the condition \eqref{Xnn2} is equivalent to 
\begin{align}
    s >\max\bigg(\Big(\frac 32 - 3 \rho \Big) \g, \, 
\frac {3 - 6 \rho + 3 (1 - \g)(1 - \g \theta)}{2 (2 - \g + (1 - \g)^2 \theta)}\bigg)
= \Big(\frac 32 - 3 \rho \Big) \g, 
\label{Xnn3}
\end{align}

\noi 
where the second equality follows
from a direct computation with \eqref{A4} and $\frac 12 < \g < 1$.
We also note from \eqref{nnn1} with \eqref{A4} and $\frac 12 < \g < 1$
 that 
 the right-hand side of~\eqref{Xnn3} is  strictly less than~$\wt S$.

\medskip

Suppose that 
$\frac 12 < \rho \le \frac 34$. 
Noting that  $\big(\frac 32 - 3 \rho \big) \g > \frac 32 - 3 \rho$, 
we obtain the range
\begin{align*}
s > \Big(\frac 32 - 3 \rho \Big) \g, 
\end{align*}

\noi
which also satisfies \eqref{reg1} in this case.
This yields (i) in \eqref{cond2}
for $\frac 12 < \rho \le \frac 34$.

Next, suppose that 
$\frac 34 < \rho < \frac 32$. 
By comparing $-\rho$ and $\big(\frac 32 - 3 \rho \big) \g$
from \eqref{reg1} and \eqref{Xnn3}, respectively, 
we obtain 
 \eqref{cond2}
 in this case.
\end{proof}

\begin{ackno}\rm
S.L.~would like to thank the School of Mathematics at the University of Edinburgh for its hospitality, where this
manuscript was prepared.
M.G.~was supported by 
the UKRI Frontier Research Grant 
(grant no.~EP/Z534328/1 ``Stochastic analysis of quantum fields"). 
D.G.~and T.O.~were supported by the European Research Council (grant no.~864138 ``SingStochDispDyn").
T.O.~was
also supported 
 by the EPSRC 
Mathematical Sciences
Small Grant  (grant no.~EP/Y033507/1)
and 
acknowledges support from  
the NSFC (grant no.~W2531005).
S.L.~was supported by the DFG through the Hausdorff Center for Mathematics under Germany's Excellence Strategy - GZ 2047/1, Project-ID 390685813, and SFB 1720, Project-ID 539309657.

\end{ackno}

\end{document}